\documentclass[reqno]{amsart}

\usepackage{fancyhdr}
\usepackage{amsmath}
\usepackage{amsthm}
\usepackage{amssymb}
\usepackage{graphicx}
\usepackage[usenames]{color}
\DeclareMathAlphabet{\mathpzc}{OT1}{pzc}{m}{it}

\DeclareFontFamily{U}{mathx}{\hyphenchar\font45}
\DeclareFontShape{U}{mathx}{m}{n}{
      <5> <6> <7> <8> <9> <10>
      <10.95> <12> <14.4> <17.28> <20.74> <24.88>
      mathx10
      }{}
\DeclareSymbolFont{mathx}{U}{mathx}{m}{n}
\DeclareFontSubstitution{U}{mathx}{m}{n}
\DeclareMathAccent{\widehat}{5}{mathx}{"70}
\DeclareMathAccent{\widecheck}{5}{mathx}{"71}
\DeclareMathAccent{\widetilde}{5}{mathx}{"72}

\newtheorem{theorem}{Theorem}[section]

\newtheorem{lemma}[theorem]{Lemma}

\newtheorem{remark}[theorem]{Remark}

\def\Holder{{H\"{o}lder}}

\def\bdy #1{\partial #1}

\def\bbN{{\mathbb N}}

\def\bbR{{\mathbb R}}
\def\bbS{{\mathbb S}}

\def\bbZ{{\mathbb Z}}

\def\div{{\operatorname{div}}}

\def\curl{{\operatorname{curl}}}

\def\id{{\text{Id}}}

\def\ft #1{{\widehat{#1}}}

\def\contsubset{\hspace{1pt}{\hookrightarrow}\hspace{1pt}}

\def\D{{\mathcal D}}

\def\N{{\mathcal N}}
\def\O{{\mathcal O}}
\def\P{{\mathcal P}}

\def\V{{\mathcal V}}

\def\rA{{\rm A}}

\def\rG{{\rm G}}
\def\rH{{\rm H}}

\def\rJ{{\rm J}}
\def\rK{{\rm K}}

\def\rN{{\rm N}}

\def\rT{{\rm T}}

\def\rg{{\rm g}}

\def\rn{{\rm n}}

\def\rr{{\rm r}}

\def\rw{{\rm w}}

\def\tv{{\tilde{v}}}

\def\tth{{{\bf h}}}
\def\ttJ{{{\bf J}}}
\def\ttq{{{\bf q}}}
\def\ttv{{{\bf v}}}

\def\Nhh{{N(\|\tth_0\|_{H^\rK(\bbS^1)}, \triplenorm{\tth}_T)}}

\def\bp{{\bar{\partial}}}
\def\p{{\partial}}

\def\Forall{\forall\hspace{2pt}}

\def\({{(\hspace{-2pt}(}}
\def\){{)\hspace{-2pt})}}

\def\triplenorm#1{{|\hspace{-1.2pt}|\hspace{-1.2pt}|{#1}|\hspace{-1.2pt}|\hspace{-1.2pt}|}}
\def\pprime{{\hspace{1pt}\prime}}

\def\XXint#1#2#3{{\setbox0=\hbox{$#1{#2#3}{\int}$}
\vcenter{\hbox{$#2#3$}}\kern-.5\wd0}}

\title[Decay for solutions of Hele-Shaw with injection]{Global existence and decay for solutions of the Hele-Shaw flow with injection}

\author[C.H. A. Cheng]{C.H. Arthur Cheng}
\email{cchsiao@math.ncu.edu.tw}
\address{Department of Mathematics, National Central University, Jhongli City, Taoyuan County, 32001, Taiwan ROC}

\author[D. Coutand]{Daniel Coutand} 
\address{CANPDE, Maxwell Institute for Mathematical Sciences and department of Mathematics, 
Heriot-Watt University, Edinburgh, EH14 4AS, UK} 
\email{D.Coutand@ma.hw.ac.uk}

\author[S. Shkoller]{Steve Shkoller}
\email{shkoller@math.ucdavis.edu}
\address{Department of Mathematics, University of California, Davis, CA 95616, USA}


\begin{document}

\begin{abstract}
We study the global existence and decay to spherical equilibrium  of Hele-Shaw flows with surface tension.   We prove that without injection of  fluid,  perturbations of the sphere  decay to zero exponentially fast.  On the other hand, with a time-dependent rate of fluid injection into the Hele-Shaw cell,  the distance from the moving boundary
to an expanding sphere (with time-dependent radius)  also decays to zero but with an algebraic rate, which  depends on the injection rate of the fluid.
\end{abstract}

\maketitle
{\small
\tableofcontents}

\section{Introduction}\label{sec_introduction}
\subsection{The problem statement}
We establish decay estimates for solutions of the Hele-Shaw equations with surface tension on the free-boundary, and with injection of fluid into the cell. With the time-dependent fluid domain denoted by $\Omega(t)$, an open subset of $\bbR^2$ with boundary $\Gamma(t)$, and for time $t \in [0,T]$, the
two-dimensional Hele-Shaw equations are given by
\begin{subequations}\label{HS_Eulerian1}
\begin{alignat}{2}
\Delta p &= -\mu\delta &&\text{in} \quad \Omega(t) \,,\\
p &= \rH \qquad&&\text{on} \quad \Gamma(t)\,,\\
\V(\Gamma(t)) &= -\frac{\p p}{\p \rn}\qquad &&\text{on}\quad\Gamma(t)\,,\\
\Omega(0) &= \Omega\,,
\end{alignat}
\end{subequations}
where $\delta$ is the Dirac delta function at the origin, $p(x,t)$ denotes the fluid pressure, $\mu = \mu(t)$ denotes the rate of injection of fluid if $\mu \ge 0$ (or suction if $\mu\le 0$), and $\rH$ is the (mean) curvature of the evolving free-boundary $\Gamma(t)$. We use $\V(\Gamma(t))$ to denote the normal velocity of the moving free-boundary $\Gamma(t)$, and we let $n$ denote the outward-pointing unit normal on $\Gamma(t)$.

When the injection rate $\mu \not\equiv 0$, the volume of the fluid domain $|\Omega(t)|$ can be computed as
\begin{align}
|\Omega(t)| = \pi + \int_0^t \mu(s) ds \equiv \pi \rho(t)^2, \label{defn:rho}
\end{align}
where $\displaystyle{}\rho(t) \equiv \sqrt{1 + \int_0^t \frac{\mu(s)}{\pi} ds} >0$ is the radius of a ball centered at the origin.
In fact, if the initial domain $\Omega = B_1 \equiv B(0,1)$, then the solution to (\ref{HS_Eulerian1}) is given explicitly by $\Omega(t) = B(0,\rho(t))$, with pressure function
$$
p(x,t) = \frac{1}{\rho(t)} + \frac{\mu(t)}{2\pi} \log \frac{|x|}{\rho(t)}.
$$
We will show that under certain growth conditions on the injection rate $\mu(t)$, if
\begin{enumerate}
\item the initial domain $\Omega$ is sufficiently close to the unit ball $B_1 \equiv B(0,1)$ in $\bbR^2$ with $|\Omega| = |B_1| = \pi$;
\item the center of mass of $\Omega$ is sufficiently close to the origin (which is the point of injection),
\end{enumerate}
then $\Omega(t)$ converges to $B_\rho \equiv B(0,\rho(t))$ as $t\to \infty$. The precise statement of our result is given below in Theorem \ref{thm:main_thm1}.

Of fundamental importance to our analysis is the conversion of the second-order Poisson equation (\ref{HS_Eulerian1}a) to a coupled system of
first-order equations.  Introducing the velocity vector
 $u = -\nabla p$, equation (\ref{HS_Eulerian1}) can be rewritten as
\begin{subequations}\label{HS_Eulerian2}
\begin{alignat}{2}
u + \nabla p &= 0 &&\text{in} \quad \Omega(t) \,,\\
\div u &=\mu \delta &&\text{in} \quad \Omega(t)\,,\\
p &= \rH \qquad&&\text{on} \quad \Gamma(t)\,,\\
\V(\Gamma(t)) &= u\cdot n \qquad&&\text{on}\quad\Gamma(t)\,,\\
\Omega(0) &= \Omega\,.
\end{alignat}
\end{subequations}

\subsection{Some prior results}
In the case that fluid is  {\it not} being injected into the Hele-Shaw cell and $\mu=0$,
Constantin \& Pugh \cite{CoPu1993} established the stability and exponential decay of solutions of (\ref{HS_Eulerian1}) using the methods of complex analysis.
Friedman and Reitich \cite{FrRe2001} also establish this stability result.
In \cite{Chen1993}, Chen studied a two-phase Hele-Shaw problem with surface tension, and established well-posedness using the energy method coupled with certain pointwise estimates from the theory of harmonic functions; moreover, he proved that solutions exist for all time if the initial interface is a sufficiently small perturbation of equilibrium.   Weak solutions have been obtained by Elliott \& Ockendon \cite{ElOc1982} and Gustafsson \cite{Gu1985}, and classical short-time solutions to related problems have been obtained by Escher and Simonett \cite{EsSi1997a} in multiple space dimensions.  In two dimensions,  Escher and Simonett \cite{EsSi1998}
establish global existence and stability near spherical shapes using center manifold theory.

In the case of fluid injection where $\mu >0$,  Prokert \cite{Pr1998} and Vondenhoff \cite{Vo2011} establish global existence results in the case that
the injection rate $\mu$ is a positive constant.

\subsection{Statement of main results}
We propose a simple methodology (equally applicable in three space dimensions) for establishing global existence and {\it decay} to equilibrium for
time-dependent injection rates. Unlike the case of zero fluid injection, the decay to equilibrium is not exponentially fast but, rather, algebraic.
 We use an Arbitrary Eulerian Lagrangian (ALE) formulation to transform the free-boundary problem (\ref{HS_Eulerian2}) to a system of PDE on a fixed domain. This ALE transformation depends on the signed height function $\tth$ measuring the signed distance between the moving surface $\Gamma(t)$ and the expanding sphere $B_{\rho(t)}$. The idea behind the proof is the construction of a (total) norm, denoted by $\triplenorm{\hspace{1.5pt}\cdot\hspace{1.5pt}}_T$, which consists of a norm $\|\cdot\|_X$ associated to
energy estimates in sufficiently high-regularity Sobolev spaces, and a norm $\|\cdot\|_Y$ associated to decay estimates in weaker topologies. The total norm is given by
\begin{align*}
\triplenorm{\hspace{1.5pt}\cdot\hspace{1.5pt}}_T = \sup_{t\in[0,T]}\left( \|\cdot\|_X + \D(t) \|\cdot\|_Y\right)
\end{align*}
for some function $\D(t) \to \infty $ as $t \to\infty $.   We shall prove that $\tth$ satisfies
\begin{align}
\triplenorm{\tth}_T \le C \epsilon + \triplenorm{\tth}_T^{2p} \label{total_norm_ineq}
\end{align}
for some integer $p$. The inequality (\ref{total_norm_ineq}) implies that if $\epsilon$ ( an upper bound of the initial for a norm of the initial data $\tth$) is sufficiently small, then $\triplenorm{\tth}_T$ stays small for all $t\in [0,T]$;  a standard continuation argument shows provides global existence and the fact $ \mathcal{D} (t)$ in front
of the lower-order norm $ \|\cdot\|_Y$ gives the decay to equilibrium.

The main results we establish in this paper are the following two theorems.
\begin{theorem}[Stability for slow injection]\label{thm:main_thm1}
Let $(p,\Omega(t))$ be the solution to {\rm(\ref{HS_Eulerian1})}, $\tth$ denote the signed distance between $\Gamma(t)=\partial \Omega(t)$ and $\bdy B_\rho$ with 
$\rho(t)$ defined by {\rm(\ref{defn:rho})}, and the center of mass of the initial domain $\Omega$ be the origin. If the injection rate $\mu\ge 0$ is such that the corresponding 
radius $\rho(t)$ satisfies
\begin{align}
\sup_{t>0} \frac{\rho^{(k)}(t)(1+t)^k}{\rho(t)} < \infty \ k=1,2, \ \text{ and }\ \sup_{t>0} \frac{\rho(t)}{(1+t)^\alpha} < \infty \text{ for some }\alpha \le \frac{1}{3}\,,
\label{rho_assumption1}
\end{align}
then there exists an $\epsilon > 0$ sufficiently small,  such that the solution to {\rm(\ref{HS_Eulerian1})} exists for all time provided that $\|\tth_0\|_{H^6(\bbS^1)} \le \epsilon$. Moreover, the signed distance $\tth$ decays to zero, and
\begin{align}
\|\tth(t)\|_{H^{2.5}(\bbS^1)} \le C \rho(t)^{-2} e^{- \beta d(t)} \qquad\Forall t>0 \label{h_decay1}
\end{align}
for some constant $C>0$ and $\beta \in (0,7/8)$, where $\displaystyle{} d(t) = \int_0^t \frac{6ds}{\rho(s)^3}$.
\end{theorem}


\begin{theorem}[Stability for fast injection]\label{thm:main_thm2}
Let $(p,\Omega(t))$ be the solution to {\rm(\ref{HS_Eulerian1})}, $\tth$ denote the signed distance between $\Gamma(t)=\partial \Omega(t)$ and $\bdy B_\rho$ with
$\rho(t)$ defined by {\rm(\ref{defn:rho})}, and $(x_0,y_0)$ be the center of mass of the initial domain $\Omega$. Suppose that $\mu\ge 0$ is such that the corresponding $\rho(t)$ satisfies
\begin{align}
\sup_{t>0} \frac{\rho^\pprime(t)(1+t)}{\rho(t)} < \infty \quad\text{and}\quad \nu \equiv \sup \Big\{ \alpha \hspace{1pt}\Big|\hspace{1pt} \sup_{t>0} \frac{(1+t)^\alpha}{\rho(t)} < \infty\Big\} > \frac{3}{8} \label{rho_assumption2}
\end{align}
as well as one of the following conditions:
\begin{enumerate}
\item $\rho^{\pprime\prime} \le 0$ or  \quad {\rm(2)} $\log \rho^\pprime$ has small enough total variation.
\end{enumerate}
Then there exists an $\epsilon > 0$ sufficiently small,  such that the solution to {\rm(\ref{HS_Eulerian1})} exists for all time provided that $\|\tth_0\|_{H^\rK(\bbS^1)} \le \epsilon$ with $\rK$ defined by $\displaystyle{}\rK = \max\Big\{6, \Big[\frac{64\nu-21}{16\nu -6}\Big]+1\Big\}$. Moreover, the signed distance $\tth$ decays to zero, and
\begin{align}
\|\tth(t)\|_{H^{2.5}(\bbS^1)} \le C \frac{\sqrt{1+t}}{\rho(t)^2} \qquad \Forall t > 0 \label{h_decay2}
\end{align}
for some constant $C>0$.
\end{theorem}
In the following discussion, we define
\begin{equation}\label{defn:D}
\D(t) = \left\{\begin{array}{cl}
\displaystyle{} \rho(t)^2 e^{\beta d(t)} & \text{if (\ref{rho_assumption1}) is satisfied}, \vspace{.2cm}\\
\displaystyle{} \frac{\rho(t)^2}{\sqrt{1+t}} & \text{if (\ref{rho_assumption2}) is satisfied}.
\end{array}
\right.
\end{equation}
Then (\ref{h_decay1}) and (\ref{h_decay2}) can be summarized as
\begin{align}
\D(t) \|\tth(t)\|_{H^{2.5}(\bbS^1)} \le C \qquad\Forall t>0
\end{align}
for some constant $C>0$.

\begin{remark}
In Theorem {\rm\ref{thm:main_thm2}}, the sufficient condition $\|\tth_0\|_{H^k(\bbS^1)} \le \epsilon$ which guarantees the existence of the global-in-time solution also suggests that the center of mass cannot be too far from the point of injection.
\end{remark}

\begin{remark}
Assumption {\rm(\ref{rho_assumption1})} limits the injection rate of the fluid. In this case, the location of the center of mass of the initial domain is crucial for the stability result, and Theorem \ref{thm:main_thm1} states that stability holds  when the center of mass of the initial domain and the point of injection are identical.

On the other hand, when {\rm(\ref{rho_assumption2})} is valid, the injection is fast enough so that even if the center of mass of the initial domain is different from (but close to) the point of injection, the stability result holds.
\end{remark}


\begin{remark}
Suppose that {\rm(\ref{rho_assumption2})} is satisfied. By Gronwall's inequality, the condition
\begin{align}
\rho'(t)\le C \rho(t) (1+t)^{-1} \quad\text{for some constant $C>0$}, \nonumber
\end{align}
 implies that $\nu < \infty$. Then for each $0<\varepsilon \ll 1$, there exists $C_1$ and $C_2$ such that
\begin{align*}
\frac{(1+t)^{\nu-\varepsilon}}{\rho(t)} \le \frac{1}{C_1} \quad \text{and}\quad \frac{(1+t)^{\nu+\varepsilon}}{\rho(t)} \ge \frac{1}{C_2} \qquad\Forall t>0
\end{align*}
so that
\begin{align}
C_1 (1+t)^{\nu-\varepsilon} \le \rho(t) \le C_2 (1+t)^{\nu+\varepsilon} \qquad\Forall t>0. \label{rho_grows_algebraically}
\end{align}
Consequently, $\rho(t)^{2-\frac{1}{2\nu^-}} \le \D(t)$, where we use the notation $\nu^\pm$ to denote $\nu\pm\varepsilon$ whenever $0<\varepsilon \ll 1$.
We will make crucial use of the integrability of certain functions of these bounds. 
\end{remark}
\vspace{.1in}

When $\mu = 0$, $\rho(t) = 1$ so (\ref{rho_assumption1}) holds. Theorem \ref{thm:main_thm1} implies that the distance from the moving boundary $\Gamma(t)$ to the boundary of the equilibrium state (wherein the boundary is the unit circle) decays to zero exponentially, which is result that was obtained by  \cite{Chen1993},\cite{CoPu1993}, \cite{FrRe2001}, and \cite{EsSi1998}, but with more conditions on the data.

When $\mu$ is a positive constant, we have $\rho(t) = \sqrt{1+\frac{\mu}{\pi}t}$.   For this case, Vondenhoff  \cite{Vo2011} provides a decay estimate
for a rescaled function:
$$
\big\|\frac{\tth}{\rho}\big\|_{H^6(\bbS^1)} \le C \rho^{-\alpha} \qquad\Forall \alpha\in (0,1)\text{ and } t>0 \,,
$$
but this, in fact, shows that $\big\|\tth\big\|_{H^6(\bbS^1)} \sim t^ {(1- \alpha )/2}$ for large $t>0$ and hence the perturbation may actually grow in time.

In contrast, for the case of constant injection, our Theorem \ref{thm:main_thm2} shows that
$$
\|\tth (t)\|_{H^{2.5}(\bbS^1)} \le C \rho(t)^{-1} \qquad\Forall t>0
$$
which is the first result proving the decay rate of the actual height function of the perturbation of the sphere.   With the rescaled variable,
$$
\big\|\frac{\tth}{\rho}\big\|_{H^6(\bbS^1)} \le C \rho(t)^{-2} \qquad\Forall t>0\,.
$$
Moreover, our theorem applies to  time-dependent injection rates, and thus generalizes the results of \cite{Pr1998} and \cite{Vo2011}.

\subsection{Outline}
In Section \ref{sec:ALE_formulation}, we derive the evolution equation for the signed height function $\tth$ and define the corresponding ALE map $\psi$, which we use to  pull-back the equations onto the fixed domain $B_1$.
In Section \ref{sec:preliminary}, we derive some inequalities, fundamental for our subsequent analysis.

In Section \ref{sec:low_norm_est}, an estimate of the type
\begin{align}
\D(t) \|\tth(t)\|_{H^{2.5}(\bbS^1)} \le C \Big[\|\tth_0\|^2_{H^{2.5}(\bbS^1)} + \triplenorm{\tth}^2_T + \triplenorm{\tth}^8_T\Big] \qquad \Forall t\in [0,T] \label{low_norm_ineq}
\end{align}
is derived using the decay property of the linearized problem, where $[0,T]$ is the time interval of the existence of the solution. We combine this with energy estimates for $\tth$ in a higher-order Sobolev spaces in Section \ref{sec:high_norm_est}, from which
we obtain
\begin{align}
&\Big[\int_0^T \frac{1}{\rho(t)^3} \|\tth(t)\|^2_{H^{\rK+1.5}(\bbS^1)} dt\Big]^\frac{1}{2}
\hspace{-2pt} +\sup_{t\in [0,T]} \Big[\|\tth(t)\|_{H^\rK(\bbS^1)} \hspace{-2pt}+ \rho(t) \sqrt{\rho^\pprime(t)} \|\tth(t)\|_{H^{\rK-1}(\bbS^1)} \Big] \nonumber \\
& \qquad\qquad\quad \le C_\delta M(\|\tth_0\|_{H^6(\bbS^1)}) + C_\delta \triplenorm{\tth}^2_T \P\big(\triplenorm{\tth}_T^2\big) + \delta \triplenorm{\tth}_T \label{high_norm_ineq}
\end{align}
for some polynomial $\P$. Finally, in Section \ref{sec:stability_result}, we show how the
decay estimate (\ref{low_norm_ineq}) together with our energy estimate (\ref{high_norm_ineq}) leads to Theorem \ref{thm:main_thm1} and \ref{thm:main_thm2}.

\section{The Arbitrary Eulerian-Lagrangian (ALE) formulation}\label{sec:ALE_formulation}
We let $\bbS^1 = \bdy B_1$ denote the boundary of the unit ball, and parametrize $\bbS^1$ using the usual angular variable $\theta$. For each $\theta\in\bbS^1$, let $\tth (\theta,t)$ denote the signed distance from $\Gamma(t)$ to $\bdy B(0,\rho(t))$, where the sign of $\tth $ is taken positive if for  ${\bf x}(\theta,t) \in \Gamma(t)$,
$|{\bf x}(\theta,t)| > \rho(t)$ and taken negative if $|{\bf x}(\theta,t)| < \rho(t)$.
 In other words, $\Gamma(t)$ can be parametrized by the equation
\begin{align}
{\bf x}(\theta,t) = (x(\theta,t),y(\theta,t)) = \big(\rho(t) + \tth (\theta,t)\big)\rN(\theta) \qquad\Forall \theta\in \bbS^1\,, \label{parametrizatioin_of_boundary}
\end{align}
where $\rN(\theta) = (\cos\theta, \sin\theta)$ is the outward-pointing unit normal to $B_1$.

Let $\rT(\theta) = (-\sin\theta,\cos\theta)$ be the tangent vector on $\bbS^1$. Then the outward-pointing unit normal at the point $(x(\theta,t),y(\theta,t))$ is
\begin{align}
n(\theta,t) = {\bf J}_{\tth}^{-1}(\theta,t) \Big[\big(\rho(t)+\tth (\theta,t)\big) \rN(\theta) - \tth_\theta(\theta,t) \rT(\theta)\Big],
\label{defn:n}
\end{align}
where $\ttJ_\tth(\theta,t) = \sqrt{(\rho(t)+\tth (\theta,t))^2 + \tth_\theta(\theta,t)^2}$.

Let $\Phi:\bbR^2 \times (0,T) \to \bbR$ denote the level-set function defined by
\begin{align*}
\Phi({\bf x},t) = \frac{1}{2} \Big[|{\bf x}|^2 - \Big(\rho(t) + \tth\big(\tan^{-1} \frac{y}{x},t\big)\Big)^2\Big].
\end{align*}
Then $\Phi$ has the property that $\Gamma(t) \equiv \big\{(x,y)\in\bbR^2\hspace{1pt}\big|\hspace{1pt} \Phi(x,y,t) = 0\big\}$. Since the boundary is moving with the fluid velocity,
$\Phi$ satisfies the transport equation
\begin{align*}
\Phi_t ({\bf x}(\theta,t), t) + u({\bf x}(\theta,t),t) \cdot (\nabla \Phi)({\bf x}(\theta,t),t) = 0 \quad \text{on}\quad \bbS^1\times (0,T)
\end{align*}
which implies that $\tth$ satisfies
\begin{align*}
- (\rho + \tth)(\rho^\pprime + \tth_t) + u({\bf x}) \cdot \Big[(\rho + \tth)\rN - \tth_\theta \rT \Big] = 0 \quad \text{on}\quad \bbS^1\times (0,T)
\end{align*}
or after rearrangement,
\begin{align}
\tth_t(\theta,t) + \rho^\pprime(t) \hspace{-1pt} = \ttv(\theta,t) \cdot \N(\theta,t) \quad \text{on}\quad \bbS^1\times (0,T), \label{heq}
\end{align}
where $\ttv(\theta,t) = u({\bf x}(\theta,t),t)$, and 
\begin{align}
\N(\theta,t) \equiv \rN(\theta) - \frac{\tth_\theta(\theta,t)}{\rho(t)+\tth(\theta,t)}\rT(\theta). \label{defn:N}
\end{align}

In addition, we note that with respect to the height function $\tth$, the mean curvature is given by
\begin{align}
\rH_\tth({\bf x}) &= \frac{- (\rho+\tth )\tth_{\theta\theta} + (\rho+\tth )^2 + 2 \tth_\theta^2}{\big[(\rho+\tth )^2 + \tth_\theta^2\big]^{3/2}} \nonumber\\
&= \frac{- (\rho + \tth) \tth_{\theta\theta} + \ttJ_\tth^2 + \tth^2_\theta}{\ttJ_\tth^3} \qquad\text{on}\quad \bbS^1\times (0,T). \label{H1}
\end{align}

\subsection{A divergence-free velocity}
When $\tth_0 = 0$ so that $\Omega = B_1$, the solution to the Hele-Shaw equation (\ref{HS_Eulerian2}) is given by $\Omega(t) = B(0,\rho(t))$,
\begin{align*}
 \bar{p}(x,t) = \frac{1}{\rho(t)} - \frac{\mu(t)}{2\pi} \log \frac{|x|}{\rho(t)}\,, \ \
 \bar{u}(x,t) = \frac{\mu(t)}{2\pi} \frac{x}{|x|^2} \,.
\end{align*}
In order to have a divergence-free velocity field, we introduce the new variables
${\bf u} = u - \bar{u}$ and ${\bf p} = p - \bar{p}$, so that (\ref{HS_Eulerian2}) is converted to
\begin{subequations}\label{HS_Eulerian3}
\begin{alignat}{2}
{\bf u} + \nabla {\bf p} &= 0 \qquad&&\text{in}\quad \Omega(t),\\
\div {\bf u} &= 0 &&\text{in}\quad\Omega(t),\\
{\bf p} &= \rH - \bar{p} \qquad&&\text{on}\quad \Gamma(t).
\end{alignat}
\end{subequations}

\subsection{The ALE formulation}
Let $\psi( \cdot ,t)$ denote the ALE  mapping, taking $B_1 $ to $\Omega(t)$, defined as the solution to the elliptic equation
\begin{subequations}\label{psi_eq1}
\begin{alignat}{2}
\Delta \psi &= 0 \qquad&&\text{in}\quad B_1,\\
\psi &= {\bf x} \qquad&&\text{on}\quad\bbS^1,
\end{alignat}
\end{subequations}
where we recall that ${\bf x} = (\rho + \tth)\rN$ on $\bbS^1$.  When the perturbation  $\tth$ is close to zero,  elliptic estimates and the inverse function theorem show that
$\psi(t)=\psi( \cdot ,t)$ is a diffeomorphism, and $\Omega(t) = \psi(t)(B_1)$. By introducing the ALE variables $\ttv = {\bf u}\circ\psi$, $\ttq = {\bf p}\circ\psi$, $\bar\ttq = \bar{p}\circ\psi$, and $\rA = \nabla \psi^{-1}$, we find that (\ref{HS_Eulerian3}) can be rewritten on the fixed domain as
\begin{subequations}\label{HS_ALE1}
\begin{alignat}{2}
\ttv^i + \rA^j_k \ttq ,_j &=0 \qquad&&\text{in}\quad B_1,\\
\rA^j_i \ttv^i,_j &=0 &&\text{in}\quad B_1,\\
\ttq &= \rH_\tth({\bf x}) - \bar{\ttq} \qquad&&\text{on}\quad \bbS^1, \\
\tth_t 
+ \rho^\pprime &= \ttv \cdot \N(\theta,t) \qquad&&\text{on}\quad \bbS^1,
\end{alignat}
\end{subequations}
where $\rH_\tth({\bf x})$ is given by (\ref{H1}) and $\N$ is defined in (\ref{defn:N}).

\subsection{The vector $\rJ \rA^\rT \rN$}
Let $\rJ = \det(\nabla \psi)$. Since $\psi(\theta,t) = \big(\rho(t)+\tth(\theta,t)\big)\rN(\theta)$ on $\bbS^1$, we find that
\begin{align}
\rJ \rA^\rT \rN &= \left[\hspace{-1.5pt}\begin{array}{cc}
\psi^2,_2 &\hspace{-1.5pt} - \psi^2,_1 \\
-\psi^1,_2 &\hspace{-1.5pt} \psi^1,_1
\end{array}\hspace{-1.5pt}
\right]\left[\hspace{-1.5pt}\begin{array}{c}
\rN_1 \\
\rN_2
\end{array}
\hspace{-1.5pt}\right] = \left[\hspace{-1.5pt}\begin{array}{cc}
\psi^2,_1 &\hspace{-1.5pt} \psi^2,_2 \\
- \psi^1,_1 &\hspace{-1.5pt} - \psi^1,_2
\end{array}
\hspace{-1.5pt}\right]\left[\hspace{-1.5pt}\begin{array}{c}
\rT_1 \\
\rT_2
\end{array}
\hspace{-1.5pt}\right] = \Big[\frac{\p \psi^2}{\p \rT}, - \frac{\p \psi^1}{\p \rT}\Big]^\rT \nonumber\\
&= \Big[\frac{\p \psi^2}{\p \theta}, - \frac{\p \psi^1}{\p \theta}\Big]^\rT = \big[ \tth_\theta \sin\theta + (\rho+\tth) \cos\theta, - \tth_\theta \cos\theta + (\rho+\tth) \sin\theta \big]^\rT \nonumber\\
&= (\rho + \tth) \rN - \tth_\theta \rT = (\rho + \tth) \N(\theta). \label{JAtN}
\end{align}

\subsection{Linearization about the unperturbed state $\tth \equiv 0$}
When $\tth_0 = 0$, $\tth = 0$ for all $t>0$, in which case, $\psi(x,t) = \rho(t)x$ and $\rA = \rho^{-1}\id$. Therefore, we may decompose (\ref{HS_ALE1}a,b) into a linear term and
a nonlinear remainder as
\begin{subequations}\label{HS_ALE2}
\begin{alignat}{2}
\ttv + \rho^{-1} \nabla \ttq &= f_1 \qquad&&\text{in}\quad B_1\times (0,T),\\
\div \ttv &= f_2 &&\text{in}\quad B_1\times (0,T),
\end{alignat}
\end{subequations}
where
\begin{align*}
f_1^i = (\rho^{-1} \delta^j_i - \rA^j_i) \ttq,_j\quad\text{and}\quad f_2 = (\delta^j_i - \rho \rA^j_i) \ttv^i,_j.
\end{align*}

In order to determine the linear operator associated to the boundary condition (\ref{HS_ALE1}c), we multiply both sides of (\ref{HS_ALE1}c) by $\rho^{-2} (\rho+\tth)^{-1} \ttJ_\tth^3$, and find that
\begin{align}
& \ttq = \frac{1}{\rho^2} \Big[- \tth_{\theta\theta} + \rho + \tth + \frac{\tth_\theta^2}{\rho+\tth} \Big] - (\gamma + 1) \bar{\ttq} - \gamma \ttq, \label{linearize_bdy_condition_temp}
\end{align}
where
\begin{align*}
\gamma &= \rho^{-2} (\rho+\tth)^{-1} \ttJ_\tth^3 - 1 = \frac{\ttJ_\tth^6 - \rho^4(\rho+\tth)^2}{\rho^2 (\rho+\tth) \big[\ttJ_\tth^3 + \rho^2(\rho+\tth)\big]} \\
&= \frac{(\rho+\tth)^2 (4\rho^3 \tth + 6\rho^2 \tth^2 + 4 \rho \tth^3 + \tth^4) + 3(\rho+\tth)^4 \tth_\theta^2 + 3 (\rho+\tth)^2 \tth_\theta^4 + \tth_\theta^6}{\rho^2 (\rho+\tth) \big[\ttJ_\tth^3 + \rho^2 (\rho+\tth)\big]}.
\end{align*}
We remark that
\begin{align*}
\gamma = \frac{2\tth}{\rho} + \rg,
\end{align*}
for some $\rg = \O(\tth^2 + \tth_\theta^2)$. 
Since
\begin{align*}
\bar\ttq = \frac{1}{\rho} - \frac{\mu}{2\pi} \log\Big(\frac{\rho + \tth}{\rho}\Big) = \frac{1}{\rho} - \frac{\mu}{2\pi} \log\Big(1+\frac{\tth}{\rho}\Big),
\end{align*}
a Taylor expansion in $\tth/\rho \ll 1$ shows that $\displaystyle{}\bar\ttq \approx \frac{1}{\rho} - \frac{\mu \tth}{2\pi\rho} = \rho^{-1} -\rho^\pprime \tth$; thus by writing $\bar{\ttq} = \bar{\bar{\ttq}} + \rho^{-1} - \rho^\pprime \tth$, where
\begin{align*}
\bar{\bar\ttq} = -\frac{\mu}{2\pi} \log \Big(1+\frac{\tth}{\rho}\Big) + \frac{\mu \tth}{2\pi\rho} = - \frac{\mu}{2\pi} \Big[\log\Big(1+\frac{\tth}{\rho}\Big) - \frac{\tth}{\rho}\Big],
\end{align*}
we may write (\ref{linearize_bdy_condition_temp}) as
$$
\ttq = -\frac{1}{\rho^2} \big[\tth_{\theta\theta} + \tth\big] + \rho^\pprime \tth - \gamma \ttq + { \mathcal{E} } \qquad\text{on}\quad \bbS^1\times (0,T), \eqno{\rm(\ref{HS_ALE2}c)}
$$
where $ \mathcal{E} $ denotes an error-term given by
\begin{align*}
{ \mathcal{E} } = \frac{h_\theta^2}{\rho^2(\rho+\tth)} - (\gamma+1) \bar{\bar{\ttq}} + \gamma \rho^\pprime \tth - \frac{\rg}{\rho}.
\end{align*}

\subsection{A homogeneous version of (\ref{HS_ALE2})}
We can now define a new velocity field which is divergence-free.
Let $f_1 = \rw + \rho^{-1} \nabla \rr$, where $\rr$ is the zero-average solution  to the following elliptic equation
\begin{subequations}\label{req1}
\begin{alignat}{2}
\rho^{-1} \Delta \hspace{1pt}\rr &= \div f_1 - f_2 \qquad&&\text{in}\quad B_1\,,\\
\rho^{-1} \frac{\p \hspace{1pt}\rr}{\p \rN} &= f_1 \cdot \rN - \frac{1}{2\pi} \int_{B_1} f_2\, dx \qquad&&\text{on}\quad\bbS^1 \,.
\end{alignat}
\end{subequations}
We remark that the solvability of (\ref{req1}) is guaranteed by the solvability condition
\begin{align*}
\int_{B_1} (\div f_1 - f_2)\, dx = \int_0^{2\pi} \Big(f_1\cdot \rN - \frac{1}{2\pi} \int_{B_1} f_2\, dx\Big) d\theta.
\end{align*}
Let $v = \ttv - \rw$, and $q = \ttq - \rr$. By (\ref{req1}a) we find that $\div \rw_1 = f_2$; thus (\ref{HS_ALE2}) implies that
\begin{subequations}\label{HS_ALE3}
\begin{alignat}{2}
v + \rho^{-1} \nabla q &= 0 \qquad&&\text{in}\quad B_1\times(0,T)\,,\\
\div v &=0 &&\text{in}\quad B_1\times(0,T)\,,\\
q &= - \frac{\tth_{\theta\theta}}{\rho^2} - \frac{\tth}{\rho^2} + \rho^\pprime \tth - \gamma \ttq + \rG \qquad&&\text{on}\quad\bbS^1\times(0,T)\,,
\end{alignat}
\end{subequations}
where $\rG = { \mathcal{E} } - \rr$. We note that on $\bbS^1$,
\begin{align}
v \cdot \N &= (\ttv-\rw)\cdot \N = \tth_t + \rho^\pprime + \frac{\tth_\theta (\rw \cdot \rT)}{\rho + \tth} -(\rw \cdot \rN) \nonumber\\
& = \tth_t + \rho^\pprime + \frac{h_\theta (\rw \cdot \rT)}{\rho + \tth} - \frac{1}{2\pi} \int_{B_1} f_2\, dx. \label{tvdotn1}
\end{align}

\section{The total norm and the basic assumptions used for our decay estimates}\label{sec:assumption}
We first define the total norm $\triplenorm{\hspace{1pt}\cdot\hspace{1pt}}_T$, used to establish (\ref{total_norm_ineq}), as follows:
\begin{equation}\label{defn:total_norm}
\begin{array}{l}
\displaystyle{}
\triplenorm{\tth}_T \equiv \Big[\int_0^T \frac{1}{\rho(t)^3} \|\tth(t)\|^2_{H^{\rK+1.5}(\bbS^1)} dt\Big]^\frac{1}{2} \vspace{.2cm}\\
\displaystyle{} \quad\ \ +\hspace{-1pt}
\sup_{t\in [0,T]} \hspace{-2pt}\Big[\|\tth(t)\|_{H^\rK(\bbS^1)} \hspace{-2pt}+ \rho(t) \sqrt{\rho^\pprime(t)} \|\tth(t)\|_{H^{K-1}(\bbS^1)} \hspace{-2pt}+ \D(t) \|\tth(t)\|_{H^{2.5}(\bbS^1)} \Big],
\end{array}
\end{equation}
where $\rK=6$ if (\ref{rho_assumption1}) is satisfied or $\displaystyle{}\rK = \Big[ \frac{64\mu-21}{16\mu-6}\Big]+1$ if (\ref{rho_assumption2}) is satisfied, and we recall that $\D(t)$ is defined in (\ref{defn:D}) by
\begin{equation*}
\D(t) = \left\{\begin{array}{cl}
\displaystyle{} \rho(t)^2 e^{\beta d(t)} & \text{if (\ref{rho_assumption1}) is satisfied}, \vspace{.2cm}\\
\displaystyle{} \frac{\rho(t)^2}{\sqrt{1+t}} & \text{if (\ref{rho_assumption2}) is satisfied}.
\end{array}
\right.
\end{equation*}
We remark that once the boundedness of $\triplenorm{\tth}_T$ is established, $\|\tth\|_{H^{2.5}(\bbS^1)}$ decays to zero at the rate $\D(t)^{-1}$.
Throughout the rest of the paper, we assume that $\rho$ satisfies the condition (\ref{rho_assumption1}) or (\ref{rho_assumption2}); that is,
\begin{equation*}
\sup_{t>0} \frac{\rho^{(k)}(t)(1+t)^k}{\rho(t)} < \infty \ k=1,2, \ \text{ and }\ \sup_{t>0} \frac{\rho(t)}{(1+t)^\alpha} < \infty \text{ for some }\alpha \le \frac{1}{3} \eqno{(\ref{rho_assumption1})}
\end{equation*}
or
\begin{equation*}
\sup_{t>0} \frac{\rho^\pprime(t)(1+t)}{\rho(t)} < \infty \quad\text{and}\quad \nu \equiv \sup \Big\{ \alpha \hspace{1pt}\Big|\hspace{1pt} \sup_{t>0} \frac{(1+t)^\alpha}{\rho(t)} < \infty\Big\} > \frac{3}{8}\,. \eqno{(\ref{rho_assumption2})}
\end{equation*}
Moreover, we make the following  basic assumption: for $t \in [0,T]$ and for sufficiently small positive constants $ \epsilon $ and $ \sigma $ to be made precise later,
\begin{align}
\|\tth(t)\|_{H^\rK(\bbS^1)} + \D(t) \|\tth(t)\|_{H^{2.5}(\bbS^1)} \le \sigma \ll 1\quad\text{and}\quad \|\tth_0\|_{H^\rK(\bbS^1)} \le  \epsilon \ll 1. \label{assumption1}
\end{align}
(We will prove that (\ref{assumption1}) indeed holds whenever the initial data is sufficiently small.)

\section{A priori estimates}\label{sec:preliminary}
\subsection{Estimates of $\psi$}
Under assumption (\ref{assumption1}), elliptic estimates show that
\begin{align}
\|\nabla \psi - \rho \id\|_{L^\infty(B_1)} &\le C \|\nabla \psi - \rho \id\|_{H^{1.5}(B_1)} \le C \|\tth\|_{H^2(\bbS^1)} 
\label{psi_est1}
\end{align}
and for $1\le s\le \rK-1$,
\begin{align}
\|D^2\psi\|_{H^s(B_1)} \le C \|\tth\|_{H^{s+1.5}(\bbS^1)}. 
\label{psi_est}
\end{align}
Estimate (\ref{psi_est1}) implies that
\begin{align*}
\|\rho\rA - \id\|_{L^\infty(B_1)} &= \|(\rho \id - \nabla \psi) \rA\|_{L^\infty(B_1)} \le C \|\rA\|_{L^\infty(B_1)} \|\tth\|_{H^2(\bbS^1)} \\
&\le C \rho^{-1} \Big[\|\rho \rA - \id\|_{L^\infty(B_1)} + 1\Big] \|\tth\|_{H^2(\bbS^1)};
\end{align*}
thus under assumption (\ref{assumption1}),
\begin{align}
\|\rho\rA - \id\|_{L^\infty(B_1)} \le C \rho^{-1} \|\tth\|_{H^2(\bbS^1)}. \label{A-1_est}
\end{align}
Note that (\ref{psi_est1}) and (\ref{A-1_est}) together imply that for $0\le s\le \rK-2$,
\begin{align}
\|D \rA\|_{H^s(B_1)} \le C \rho^{-2} \|\tth\|_{H^{s+1.5}(\bbS^1)} 
. \label{A_est2}
\end{align}
Furthermore, with $\rJ \equiv \det(\nabla \psi)$, inequality (\ref{psi_est1})  implies that
\begin{align}
\|\rJ - \rho^2\|_{L^\infty(B_1)} \le C \|\tth\|_{H^2(\bbS^1)} \label{J-1_est}
\end{align}
and (\ref{A-1_est}) and (\ref{A_est2}) together with (\ref{J-1_est}) shows that
\begin{align}
\|D \rJ\|_{H^s(B_1)} \le C \rho \|\tth\|_{H^{s+1.5}(\bbS^1)}\,, \label{J_est2}
\end{align}
from which it follows that for $0\le s\le \rK-2$,
\begin{align}
\|D (\rJ \rA)\|_{H^s(B_1)} \le C \|\tth\|_{H^{s+1.5}(\bbS^1)}. \label{JA_est}
\end{align}

\subsection{Estimates of $\bar{\ttq}$ and $\bar{\bar{\ttq}}$}
Recall that $\displaystyle{}\bar{\bar\ttq} = - \frac{\mu}{2\pi} \Big[\log\Big(1+\frac{\tth}{\rho}\Big) - \frac{\tth}{\rho}\Big]$. Since $\big|\log(1+x) - x\big| \le x^2$ if $|x| < 0.5$, and
\begin{align*}
\bp \bar{\bar\ttq} = \frac{\rho^\pprime \tth \bp \tth}{\rho+\tth} ,
\end{align*}
we find that with the basic assumption (\ref{assumption1}), for $1\le s\le \rK-2.5$,
\begin{align}
\|\bar{\bar{\ttq}}\|_{H^s(\bbS^1)} \le C \rho^\pprime \rho^{-1} \|\tth\|_{H^2(\bbS^1)} \|\tth\|_{H^s(\bbS^1)} 
. \label{qbb_est}
\end{align}
Since $\bar{\ttq} = \bar{\bar{\ttq}} + \rho^{-1} - \rho^\pprime \tth$,
\begin{align}
\|\bar{\ttq}\|_{H^s(\bbS^1)} \le C \Big[\rho^{-1} \big(1 + \|\tth\|_{H^2(\bbS^1)} \|\tth\|_{H^s(\bbS^1)}\big) + \rho^\pprime \|\tth\|_{H^s(\bbS^1)} \Big]. \label{qb_est}
\end{align}

\subsection{Estimates of $f_1$ and $f_2$}
By the definition of $f_1$ and $f_2$, for $s\ge 1$,
\begin{align}
\|f_1\|_{H^s(B_1)} \le C \rho^{-2} \Big[\|\tth\|_{H^{s+1}(\bbS^1)} \|\ttq\|_{H^{1.5}(B_1)} + \|\tth\|_{H^2(\bbS^1)} \|\ttq\|_{H^{s+1}(B_1)} \Big] \label{f1_est_temp}
\end{align}
and
\begin{align}
\|f_2\|_{H^{s-1}(B_1)} \le C \rho^{-1} \Big[\|\tth\|_{H^{s+1}(\bbS^1)} \|\ttv\|_{H^{1.5}(B_1)} + \|\tth\|_{H^2(\bbS^1)} \|\ttv\|_{H^s(B_1)} \Big]. \label{f2_est_temp}
\end{align}

\subsection{Estimates of $\ttq$}
Before proceeding to the estimate of $\ttq$, we remark that since
$$\displaystyle{}{ \mathcal{E} } = \frac{h_\theta^2}{\rho^2(\rho+\tth)} - (\gamma+1) \bar{\bar{\ttq}} + \gamma \rho^\pprime \tth - \frac{1}{\rho} \rg\,,$$
by (\ref{qbb_est}) we obtain that
\begin{align}
\|{ \mathcal{E} }\|_{H^{s+0.5}(\bbS^1)} \le C \frac{1+\rho^\pprime}{\rho} \|\tth\|_{H^2(\bbS^1)} \Big[\|\tth\|_{H^{s+0.5}(\bbS^1)} + \rho^{-2} \|\tth\|_{H^{s+2.5}(\bbS^1)} \Big]. \label{g1_est}
\end{align}
Note that $\ttq$ satisfies
\begin{subequations}\label{ttq_eq}
\begin{alignat}{2}
\Delta \ttq &= \rho \big[\div f_1 - f_2\big] \qquad&&\text{in}\quad B_1\times (0,T),\\
\ttq &= - \frac{\tth_{\theta\theta}}{\rho^2} - \frac{\tth}{\rho^2} + \rho^\pprime \tth - \gamma \ttq + { \mathcal{E} } \qquad&&\text{on}\quad \bbS^1\times (0,T).
\end{alignat}
\end{subequations}
Elliptic estimates together with (\ref{qb_est})-(\ref{g1_est}) then show that
\begin{align*}
\|\ttq\|_{H^2(B_1)} &\le C \Big[ \rho \big(\|f_1\|_{H^1(B_1)} \hspace{-2pt}+ \|f_2\|_{L^2(B_1)}\big) \hspace{-2pt}+ \rho^{-2} \|\tth\|_{H^{3.5}(\bbS^1)} \hspace{-2pt}+ \rho^\pprime \|\tth\|_{H^{1.5}(\bbS^1)} \\
&\qquad \hspace{-2pt}+ \|\gamma \ttq\|_{H^{1.5}(\bbS^1)} \hspace{-2pt}+ \|{ \mathcal{E} }\|_{H^{1.5}(\bbS^1)} \Big] \\
&\le C \Big[\|\tth\|_{H^2(\bbS^1)}\big(\|\ttq\|_{H^2(B_1)} \hspace{-2pt}+ \|\ttv\|_{H^1(B_1)}\big) \hspace{-2pt}+ \rho^{-2} \|\tth\|_{H^{3.5}(\bbS^1)} \hspace{-2pt}+ \rho^\pprime \|\tth\|_{H^{1.5}(\bbS^1)} \\
&\qquad \hspace{-2pt}+ \|\tth\|_{H^2(\bbS^1)} \|\tth\|_{H^{1.5}(\bbS^1)} \hspace{-2pt}+ \rho^{-2} \|\tth\|_{H^2(\bbS^1)} \|\tth\|_{H^{2.5}(\bbS^1)} \hspace{-2pt}+ \|\gamma \ttq\|_{H^{1.5}(\bbS^1)} \Big].
\end{align*}
Since
\begin{align*}
\|\ttv\|_{H^1(B_1)} &\le \rho^{-1} \|\ttq\|_{H^2(B_1)} + \|f_1\|_{H^1(B_1)} \le C \big(\rho^{-1} + \|\tth\|_{H^2(\bbS^1)}\big) \|\ttq\|_{H^2(B_1)}, \\
\|\gamma \ttq\|_{H^{1.5}(\bbS^1)} &\le C \Big[\rho^{-1} \|\tth\|_{H^{1.5}(\bbS^1)} + \rho^{-2} \|\tth\|^2_{H^{2.5}(\bbS^1)}\Big] \|\ttq\|_{H^2(B_1)},
\end{align*}
by assumption (\ref{assumption1}) we find that
\begin{align*}
\|\ttq\|_{H^2(B_1)} &\le C \Big[\rho^{-2} \|\tth\|_{H^{3.5}(\bbS^1)} + \rho^\pprime \|\tth\|_{H^{1.5}(\bbS^1)} \Big]. 
\end{align*}
Similarly, elliptic estimates also show that for integers $1\le k \le K-1$,
\begin{align*}
\|\ttq\|_{H^{k+1}(B_1)}
&\le C \Big[\rho^{-1} \|\tth\|_{H^{k+1}(\bbS^1)} \big(\|\ttq\|_{H^{1.5}(B_1)} + \|\ttv\|_{H^{1.5}(B_1)}\big) \\
&\qquad + \rho^{-1} \|\tth\|_{H^2(\bbS^1)} \big(\|\ttq\|_{H^{k+1}(B_1)} + \|\ttv\|_{H^k(B_1)} \big) \\
&\qquad + \rho^{-2} \|\tth\|_{H^{k+2.5}(\bbS^1)} + \rho^\pprime \|\tth\|_{H^{k+0.5}(\bbS^1)} + \|\tth\|_{H^2(\bbS^1)} \|\tth\|_{H^k(\bbS^1)} \\
&\qquad + \rho^{-2} \|\tth\|_{H^2(\bbS^1)} \|\tth\|_{H^{k+1.5}(\bbS^1)}\Big] + C \|\gamma \ttq\|_{H^{k+0.5}(\bbS^1)}.
\end{align*}
Since
\begin{align*}
\|\ttv\|_{H^s(B_1)} &\le \rho^{-1} \|\ttq\|_{H^{s+1}(B_1)} + \|f_1\|_{H^s(B_1)} \\
&\le C \rho^{-1} \Big[\|\ttq\|_{H^{s+1}(B_1)} + \|\tth\|_{H^{3.5}(\bbS^1)} \|\tth\|_{H^{s+1}(\bbS^1)} \Big],
\end{align*}
and
\begin{align}
\|\gamma \ttq\|_{H^{s+0.5}(\bbS^1)} &\le C \Big[\|\gamma\|_{H^1(\bbS^1)} \|\ttq\|_{H^{s+1}(B_1)} + \|\gamma\|_{H^{s+0.5}(\bbS^1)} \|\ttq\|_{H^{1.5}(B_1)} \Big] \nonumber\\
&\le C \Big[\big(\rho^{-1} \|\tth\|_{H^1(\bbS^1)} + \rho^{-2} \|\tth\|^2_{H^2(\bbS^1)}\big) \|\ttq\|_{H^{s+1}(B_1)} \label{gq_est}\\
&\qquad + \rho^{-1} \|\tth\|_{H^{s+0.5}(\bbS^1)} + \rho^{-2} \|\tth\|_{H^2(\bbS^1)} \|\tth\|_{H^{s+1.5}(\bbS^1)} \|\tth\|_{H^{3.5}(\bbS^1)} \Big], \nonumber
\end{align}
we conclude that
\begin{align*}
\|\ttq\|_{H^{k+1}(B_1)} &\le C \Big[\rho^{-2} \|\tth\|_{H^{k+1}(\bbS^1)} \|\ttq\|_{H^{3}(B_1)} + \rho^{-1} \|\tth\|_{H^{3.5}(\bbS^1)} \|\tth\|_{H^{k+1}(\bbS^1)} \\
&\qquad + \rho^{-2} \|\tth\|_{H^{k+2.5}(\bbS^1)} + \rho^\pprime \|\tth\|_{H^{k+0.5}(\bbS^1)} + \|\tth\|_{H^2(\bbS^1)} \|\tth\|_{H^k(\bbS^1)} \\
&\qquad + \rho^{-2} \|\tth\|_{H^2(\bbS^1)} \|\tth\|_{H^{k+1.5}(\bbS^1)} \Big] .
\end{align*}
In particular, when $k=2$, using our basic assumption (\ref{assumption1}), we obtain that
\begin{align*}
\|\ttq\|_{H^{3}(B_1)} &\le C \Big[\rho^{-1} \|\tth\|_{H^{3.5}(\bbS^1)} \|\tth\|_{H^{3}(\bbS^1)} \\
&\qquad + \rho^{-2} \|\tth\|_{H^{4.5}(\bbS^1)} + \rho^\pprime \|\tth\|_{H^{2.5}(\bbS^1)} + \|\tth\|_{H^2(\bbS^1)}^2 \\
&\qquad + \rho^{-2} \|\tth\|_{H^2(\bbS^1)} \|\tth\|_{H^{3.5}(\bbS^1)} \Big]
\end{align*}
and is hence made small by (\ref{assumption1}).   This,
 in turn,  implies that
\begin{align}
&\|\ttv\|_{H^k(B_1)} + \rho^{-1} \|\ttq\|_{H^{k+1}(B_1)} \le C \Big[ \rho^{-2} \|\tth\|_{H^{4.5}(\bbS^1)} \|\tth\|_{H^{k+1}(\bbS^1)} + \rho^{-3} \|\tth\|_{H^{k+2.5}(\bbS^1)} \nonumber\\
&\qquad\qquad\quad + \frac{\rho^\pprime}{\rho} \|\tth\|_{H^{k+0.5}(\bbS^1)} + \rho^{-3} \|\tth\|_{H^2(\bbS^1)} \|\tth\|_{H^{k+1.5}(\bbS^1)} \Big] \nonumber\\
&\qquad\qquad \le C \rho^{-1} \Big[\|\tth\|_{H^{k+2.5}(\bbS^1)} + \rho^\pprime \|\tth\|_{H^{k+0.5}(\bbS^1)}\Big]. \label{ttvq_est}
\end{align}
Consequently,
\begin{equation}\label{f1f2_est}
\begin{array}{l}
\displaystyle{} \|f_1\|_{H^s(B_1)} + \|f_2\|_{H^{s-1}(B_1)} \vspace{.2cm}\\
\displaystyle{}\qquad\qquad\ \ \le C \rho^{-2}
 \Big[ \|\tth\|_{H^{s+2.5}(\bbS^1)} + \rho^\pprime \|\tth\|_{H^{s+0.5}(\bbS^1)} \Big]
\end{array}
\end{equation}
for $1\le s\le \rK-2$, where we have used the basic assumption $ \|\tth\|_{H^{2.5}(\bbS^1)} \le \sigma $.

\subsection{Estimates of $\rr$ and $\rw$}
Applying elliptic estimates to (\ref{req1}), using (\ref{f1f2_est}) we find that
\begin{align}
\|\rr\|_{H^{s+1}(B_1)} &\le C \rho^{-1} \Big[\|f_1\|_{H^s(B_1)} + \|f_2\|_{H^{s-1}(B_1)} \Big] \nonumber\\
&\le C \rho^{-3}\|\tth\|_{H^{2.5}(\bbS^1)} \Big[ \|\tth\|_{H^{s+2.5}(\bbS^1)} + \rho^\pprime \|\tth\|_{H^{s+0.5}(\bbS^1)} \Big]. \label{r2_est}
\end{align}
Since $\rw = f_1 - \nabla \rr$, we also obtain that
\begin{align}
\|\rw\|_{H^s(B_1)} &\le \|f_1\|_{H^s(B_1)} + \|\rr\|_{H^{s+1}(B_1)} \nonumber\\
&\le C \rho^{-2} \|\tth\|_{H^{2.5}(\bbS^1)} \Big[ \|\tth\|_{H^{s+2.5}(\bbS^1)} + \rho^\pprime \|\tth\|_{H^{s+0.5}(\bbS^1)} \Big].
. \label{rw1_est}
\end{align}

\subsection{Estimates of $v$ and $q$ in terms of $\tth$}
Since $v = \ttv - \rw$ and $q = \ttq - \rr$,
\begin{align}
&\|v\|_{H^s(B_1)} + \rho^{-1} \|q\|_{H^{s+1}(B_1)} \nonumber\\
&\qquad\quad \le \|\ttv\|_{H^s(B_1)} + \|\rw\|_{H^s(B_1)} + \rho^{-1} \|\ttq\|_{H^{s+1}(B_1)} + \rho^{-1} \|\rr\|_{H^{s+1}(B_1)} \nonumber\\
&\qquad\quad \le C \rho^{-1} \Big[\|\tth\|_{H^{s+2.5}(\bbS^1)} + \rho^\pprime \|\tth\|_{H^{s+0.5}(\bbS^1)}\Big]. \label{vq_in_terms_of_h}
\end{align}
Moreover, since $\nabla q = \rho v = \rho \ttv - \rho \rw$,
\begin{align*}
\|q\|_{H^{s+1}(B_1)} &\le \rho \big[\|\ttv\|_{H^s(B_1)} + \|\rw\|_{H^s(B_1)} \big] + \|q\|_{H^2(B_1)} \nonumber\\
&\le \rho \|\ttv\|_{H^s(B_1)} + C \Big[\|\tth\|_{H^{2.5}(\bbS^1)} + \rho^\pprime \|\tth\|_{H^{0.5}(\bbS^1)}\Big]
\end{align*}
which together with (\ref{r2_est}) further implies that
\begin{align*}
\|\ttq\|_{H^{s+1}(B_1)} &\le \rho \|\ttv\|_{H^s(B_1)} + C \Big[\|\tth\|_{H^{2.5}(\bbS^1)} + \rho^\pprime \|\tth\|_{H^{0.5}(\bbS^1)}\Big] \\
&\quad + C \rho^{-3} \|\tth\|_{H^{2.5}(\bbS^1)} \Big[ \|\tth\|_{H^{s+2.5}(\bbS^1)} + \rho^\pprime \|\tth\|_{H^{s+0.5}(\bbS^1)} \Big]. \nonumber
\end{align*}
Therefore,
\begin{equation}\label{ttq_in_terms_of_ttv}
\begin{array}{l}
\displaystyle{} \|\ttq\|_{H^{s+1}(B_1)} \vspace{.2cm}\\
\displaystyle{}\quad\ \le \rho \|\ttv\|_{H^s(B_1)} + C (1+\rho^\pprime) \Big[\|\tth\|_{H^{2.5}(\bbS^1)} + \rho^{-3} \|\tth\|_{H^{2.5}(\bbS^1)} \|\tth\|_{H^{s+2.5}(\bbS^1)}\Big].
\end{array}
\end{equation}
On the other hand, we also have
\begin{align*}
&\|\nabla \ttq\|_{H^s(B_1)} \le \|(\rho \rA^\rT - \id) \nabla \ttq\|_{H^s(B_1)} + \rho \|\rA^\rT \nabla \ttq\|_{H^s(B_1)} \nonumber\\
&\ \le \rho \|\ttv\|_{H^s(B_1)} \hspace{-2pt} + C \Big[\|\id - \rho\rA\|_{L^\infty(\bbS^1)} \|\nabla\ttq\|_{H^s(B_1)} \hspace{-2pt} + \rho \|D \rA\|_{W^{s-1,4}(B_1)} \|\nabla \ttq\|_{L^4(B_1)}\Big];
\end{align*}
thus by estimates (\ref{A-1_est}), (\ref{A_est2}) and assumption (\ref{assumption1}) we find that for $1\le s\le \rK-1$,
\begin{align}
\|\nabla \ttq\|_{H^{s+1}(B_1)} \le \rho \|\ttv\|_{H^s(B_1)}. \label{ttq_in_terms_of_ttv1}
\end{align}

\subsection{Elliptic estimates for the height function $\tth$}
We may rewrite the boundary condition (\ref{HS_ALE3}c) as
\begin{align*}
- \tth_{\theta\theta} + \rho^2 \rho^\pprime \tth = \tth + \rho^2 \big[\ttq + \gamma \ttq - { \mathcal{E} } \big].
\end{align*}
Elliptic estimates then imply that for $1\le s\le K-1$,
\begin{align*}
&\|\tth\|^2_{H^{s+2.5}(\bbS^1)} + \rho^2 \rho^\pprime \|\tth\|^2_{H^{s+1.5}(\bbS^1)} \\
&\qquad \le C \|\tth\|^2_{H^{s+0.5}(\bbS^1)} + \rho^4 \Big[\|\ttq\|^2_{H^{s+1}(B_1)} + \|\gamma \ttq\|^2_{H^{s+1}(B_1)} + \|{ \mathcal{E} }\|^2_{H^{s+0.5}(\bbS^1)}\Big];
\end{align*}
thus by (\ref{g1_est}), (\ref{gq_est}), (\ref{r2_est}) and (\ref{ttq_in_terms_of_ttv}) together with the smallness of $\|\tth\|_{H^6(\bbS^1)}$,
 as well as the boundedness of $\rho^\pprime/\rho$, we find that
\begin{align*}
& \|\tth\|^2_{H^{s+2.5}(\bbS^1)} + \rho^2 \rho^\pprime \|\tth\|^2_{H^{s+1.5}(\bbS^1)} \le C \Big[\rho^6 \|\ttv\|^2_{H^s(B_1)} + \|\tth\|^2_{H^{s+1.5}(\bbS^1)} \\
&\qquad\qquad\quad + \rho^2 \|\tth\|^2_{H^{s+0.5}(\bbS^1)} + \rho^4 (1+\rho^{\pprime 2}) \|\tth\|^2_{H^{2.5}(\bbS^1)} \Big].
\end{align*}
It then follows that for $1\le s\le \rK-1$,
\begin{align}
& \rho^{-3} \|\tth\|^2_{H^{s+2.5}(\bbS^1)} 
\le C \Big[\rho^3 \|\ttv\|^2_{H^s(B_1)} \hspace{-2pt}+ \big(\rho^{-3} + \rho \rho^{\pprime 2}\big) \|\tth\|^2_{H^{2.5}(\bbS^1)} \Big].
\label{elliptic_est1}
\end{align}

\subsection{An elliptic estimate via the Hodge decomposition}
\begin{lemma}\label{lem:hodge}
Suppose that $v\in L^2(B_1)$ satisfies $\div v = \curl v =0$, and the tangential derivatives $\bp^\ell v \in L^2(B_1)$ for all $\ell=1,2\cdots,k$. Then $v\in H^k(B_1)$ and satisfies
\begin{align}
\|v\|_{H^k(B_1)} 
&\le C \sum_{j=0}^k \|\bp^j v\|_{L^2(B_1)}. \label{hodge}
\end{align}
\end{lemma}
\begin{proof}
For $k \in \bbN$,
\begin{align*}
\|v\|_{H^k(B_1)} \le C \Big[\|v\|_{L^2(B_1)} \hspace{-2pt}+ \|\curl v\|_{H^{k-1}(B_1)} \hspace{-2pt}+ \|\div v\|_{H^{k-1}(B_1)} \hspace{-2pt}+ \|v\cdot \rN\|_{H^{k+0.5}(\bbS^1)} \Big].
\end{align*}
Since $\curl v = \div v = 0$, we find that
\begin{align*}
&\|v\|_{H^k(B_1)} \le C \Big[\|v\|_{L^2(B_1)} \hspace{-2pt}+ \|v \cdot \rN\|_{H^{k-0.5}(\bbS^1)} \Big] \\
&\quad\le C \Big[\|v\|_{L^2(B_1)} \hspace{-2pt}+ \|\bp (v\cdot \rN)\|_{H^{k - 1.5}(\bbS^1)} \Big] \\
&\quad\le C \Big[\|v\|_{L^2(B_1)} \hspace{-2pt}+ \|\bp v \cdot \rN\|_{H^{k - 1.5}(\bbS^1)} \hspace{-2pt}+ \|v \cdot \rT\|_{H^{k-1.5}(\bbS^1)} \Big] \\
&\quad\le C \Big[\|v\|_{L^2(B_1)} \hspace{-2pt}+ \|\bp^2 v\cdot \rN\|_{H^{k-2.5}(\bbS^1)} \hspace{-2pt}+ \|\bp v\cdot \rT\|_{H^{k - 2.5}(\bbS^1)} \hspace{-2pt}+ \|\bp(v\cdot \rT) \|_{H^{k-2.5}(\bbS^1)} \Big] \\
&\quad\le C \Big[\|v\|_{L^2(B_1)} \hspace{-2pt}+ \|\bp v\cdot \rN\|_{H^{k - 2.5}(\bbS^1)} \hspace{-2pt}+ \|\bp^2 v\cdot \rN\|_{H^{k-2.5}(\bbS^1)} \hspace{-2pt}+ \|\bp v \cdot \rT\|_{H^{k-2.5}(\bbS^1)} \Big] \\
&\quad\le C \Big[\|v\|_{L^2(B_1)} \hspace{-2pt}+ \sum_{j=1}^k \|\bp^j v\cdot \rN\|_{H^{-0.5}(\bbS^1)} \hspace{-2pt}+ \|\bp v\cdot \rT\|_{H^{-0.5}(\bbS^1)} \Big] \,.
\end{align*}
Again, with $\div v = \curl v =0$, we see that $\|v\cdot \rN\|_{H^{k+0.5}(\bbS^1)} \le C \|v\|_{H^k(B_1)}$
\end{proof}

\section{Decay estimates for $\|\tth\|_{H^{2.5}(\bbS^1)}$}\label{sec:low_norm_est}
\subsection{The Fourier representation of solutions}
Let $\rG_1 = \rG + \gamma \ttq$, and express $\tth$ and $\rG_1$ in terms of their Fourier series:
\begin{align*}
\tth(\theta,t) = \frac{1}{\sqrt{2\pi}} \sum\limits_{k\in \bbZ} \ft{\tth}_k(t) e^{ik\theta}, \quad \rG_1(\theta,t) = \frac{1}{\sqrt{2\pi}} \sum\limits_{k\in\bbZ} \ft{\rG_1}_k(t) e^{ik\theta}. 
\end{align*}
Since $q$ is harmonic in $B_1$ with the Dirichlet boundary condition (\ref{HS_ALE3}c), the Poisson integral formula shows that
\begin{align*}
q(r,\theta,t) 
= \frac{1}{\sqrt{2\pi}} \sum_{k\in \bbZ} \Big[\frac{|k|^2-1}{\rho^2} \ft{\tth}_k + \rho^\pprime \ft{\tth}_k + \ft{\rG_1}_k \Big] r^{|k|} e^{ik\theta} \quad \text{for}\quad r<1\,.
\end{align*}
Taking the inner product of (\ref{HS_ALE3}) and the vector $\N$, by (\ref{tvdotn1}) and that $\displaystyle{}\frac{\p q}{\p \rN} = \frac{\p q}{\p r}\Big|_{r=1}$, we find that 
\begin{align*}
&\tth_t + \rho^\pprime + \frac{1}{\rho} \frac{\p q}{\p r}\Big|_{r=1} = \frac{h_\theta \nabla q\cdot\rT}{\rho(\rho + \tth)} \Big|_{r=1} - \frac{h_\theta (\rw \cdot \rT)}{\rho + \tth}\Big|_{r=1} + \frac{1}{2\pi} \int_{B_1} f_2\, dx
\end{align*}
or 
\begin{equation}\label{hk_id_temp}
\begin{array}{l}
\displaystyle{} \tth_t + \frac{1}{\sqrt{2\pi}} \sum_{k\in\bbZ} \Big[\frac{|k|(|k|^2-1)}{\rho^3} + |k| \frac{\rho^\pprime}{\rho} \Big]\ft{\tth}_k e^{ik\theta} = - \frac{1}{\sqrt{2\pi} \rho}\sum_{k\ne 0} |k| \ft{\rG_1}_k e^{ik\theta} \vspace{.2cm}\\
\displaystyle{} \hspace{50pt} + \frac{h_\theta}{\rho(\rho+\tth)} \frac{\p q}{\p \theta} \Big|_{r=1} - \frac{h_\theta (\rw \cdot \rT)}{\rho+\tth}\Big|_{r=1} + \frac{1}{2\pi} \int_{B_1} f_2\, dx.
\end{array}
\end{equation}
Therefore, for $k\ne 0,\pm 1$ we have
\begin{align}
\frac{d\ft{\tth}_k}{dt} + \Big[\frac{|k|(|k|^2-1)}{\rho^3} + |k|\frac{\rho^\pprime}{\rho} \Big] \ft{\tth}_k = \ft{R}_k, \label{hk_id1}
\end{align}
where $\ft{R}_k$ is the Fourier coefficient of $R$ defined as the right-hand side of (\ref{hk_id_temp}).
Note that estimates (\ref{f2_est_temp}) and (\ref{r2_est}) then suggest that
\begin{align}
\|R\|_{H^{2.5}(\bbS^1)} &\le C \Big[\|\rG + \gamma \ttq\|_{H^{3.5}(\bbS^1)} + \rho^{-1} \|\tth\|_{H^{2.5}(\bbS^1)} \|q\|_{H^4(B_1)} \nonumber\\
&\qquad + \rho^{-1} \|\tth\|_{H^{3.5}(\bbS^1)} \|q\|_{H^{2.5}(B_1)} \Big] \nonumber\\
&\le C \rho^{-1} \|\tth\|_{H^{2.5}(\bbS^1)} \Big[\rho^{-1} \|\tth\|_{H^{5.5}(\bbS^1)} + \rho^\pprime \|\tth\|_{H^{3.5}(\bbS^1)} \Big]. \label{R_est_insu}
\end{align}

\subsection{The $H^{2.5}$-decay estimate}\label{sec:H25_decay}
Define
\begin{align*}
I_k(t) &= \int_0^t \Big[\frac{|k|(|k|^2-1)}{\rho^3} + |k| \frac{\rho^\pprime}{\rho}\Big](t') dt' = \int_0^t \frac{|k|(|k|^2-1)}{\rho(t')^3}\, dt' + |k|\log \rho(t) \,.
\end{align*}
Then the use of $I_k(t)$ as the integrating factor in (\ref{hk_id1}) implies that
\begin{align*}
\ft{\tth}_k(t) = e^{-I_k(t)} \ft{\tth_0}_k + \int_0^t e^{-I_k(t) + I_k(s)} \ft{R}_k(s) ds
\end{align*}
Since $d(t) + 2\log \rho(t) = I_2(t) \le I_k(t)$ for all $|k|\ge 2$, and $\rho(s)\le \rho(t)$ for all $s\le t$, we find that
\begin{align*}
|\ft{\tth}_k(t)| \le 
\frac{e^{-d(t)}}{\rho(t)^2} \Big[|\ft{\tth_0}_k| + \int_0^t \rho(s)^2 e^{d(s)} |\ft{R}_k(s)| ds \Big]\qquad \Forall |k|\ge 2; 
\end{align*}
thus by \Holder's inequality,
\begin{equation}\label{decay_estimate_insu}
\begin{array}{l}
\displaystyle{} \sum_{k\ne 0,\pm 1} (1+|k|^5) |\ft{\tth}_k(t)|^2 \\
\displaystyle{} \qquad\quad \le \frac{C e^{-2d(t)}}{\rho(t)^4} \Big[\|\tth_0\|^2_{H^{2.5}(\bbS^1)} + t \int_0^t \rho(s)^4 e^{2 d(s)} \|R(s)\|^2_{H^{2.5}(\bbS^1)} ds\Big] .
\end{array}
\end{equation}
Since $\rho^\pprime(t)$ is bounded, by (\ref{R_est_insu}) and interpolation we obtain that
\begin{align}
& \int_0^t \rho(s)^4 e^{2 d(s)} \|R(s)\|^2_{H^{2.5}(\bbS^1)} ds \label{decay_estimate_process}\\
& \qquad \le C \hspace{-1pt}\int_0^t \hspace{-2pt} e^{2d(s)} \|\tth(s)\|^2_{H^{2.5}(\bbS^1)} \Big[\|\tth(s)\|^2_{H^{5.5}(\bbS^1)} \hspace{-2pt}+ \rho(s)^2 \rho^\pprime(s)^2 \|\tth(s)\|^2_{H^{3.5}(\bbS^1)} \Big] ds. \nonumber
\end{align}

\subsubsection{The case that $\rho$ satisfies {\rm(\ref{rho_assumption1})}}
In this case, by interpolation (\ref{decay_estimate_process}) implies that
\begin{align*}
& \int_0^t \rho(s)^4 e^{2 d(s)} \|R(s)\|^2_{H^{2.5}(\bbS^1)} ds \le C \hspace{-1pt}\int_0^t \|\tth(s)\|^{16/7}_{H^{2.5}(\bbS^1)} \|\tth(s)\|^{12/7}_{H^6(\bbS^1)} ds \\
&\qquad\quad + C \int_0^t \rho(s)^2 \rho^\pprime(s)^2 \|\tth(s)\|^{16/5}_{H^{2.5}(\bbS^1)} \|\tth(s)\|^{4/5}_{H^5(\bbS^1)} ds.
\end{align*}
Since $\displaystyle{} \|\tth(t)\|_{H^6(\bbS^1)} + \rho(t)^2 e^{\beta d(t)} \|\tth(t)\|_{H^{2.5}(\bbS^1)} \le \triplenorm{\tth}_T$, we find that
\begin{align*}
&\int_0^t \rho(s)^4 e^{2d(s)} \|R(s)\|^2_{H^{2.5}(\bbS^1)} ds \\
&\qquad\qquad \le C \Big[\int_0^t \Big(\frac{e^{(2-\frac{16}{7} \beta) d(s)}}{\rho(s)^{32/7}} + \frac{\rho^\pprime(s)^2 e^{(2- \frac{24}{7}\beta) d(s)}}{\rho(s)^{34/7}}\Big) ds \Big] \triplenorm{\tth}^4_T \\
&\qquad\qquad \le C \Big[\int_0^t \Big(\frac{1}{\rho(s)^{11/7}} + \frac{\rho^\pprime(s)^2}{\rho(s)^{13/7}} \Big) e^{(2-\frac{16}{7} \beta) d(s)} d^\pprime(s) \Big) ds\Big] \triplenorm{\tth}^4_T;
\end{align*}
thus by (\ref{rho_assumption1}), we obtain that
\begin{align*}
& \int_0^t \rho(s)^4 e^{2d(s)} \|R(s)\|^2_{H^{2.5}(\bbS^1)} ds \le C \int_0^t e^{(2-\frac{16}{7}\beta) d(s)} d^\pprime(s) ds \\
&\qquad\qquad = \frac{C}{2 - \frac{16}{7}\beta}\, e^{(2-\frac{16}{7}\beta) d(s)} \Big|_{s=0}^{s=t} \le C_\beta e^{(2 - \frac{16}{7}\beta) d(t)}
\end{align*}
if $\displaystyle{}2 - \frac{16}{7}\beta > 0$ or $\displaystyle{}\beta < \frac{7}{8}$\,. Inequality (\ref{decay_estimate_insu}) then implies that
\begin{align}
\sum_{|k|\ge 2} (1+|k|^5) |\ft{\tth}_k(t)|^2 &\le C \frac{e^{-2d(t)}}{\rho(t)^4} \Big[\|\tth_0\|^2_{H^{2.5}(\bbS^1)} + C_\beta t e^{(2-\frac{16}{7}\beta) d(t)} \triplenorm{\tth}^4_T \Big] \nonumber\\
&\le C \frac{e^{-2\beta d(t)}}{\rho(t)^4} \Big[\|\tth_0\|^2_{H^{2.5}(\bbS^1)} + C_\beta \triplenorm{\tth}^4_T \Big], \label{decay_estimate_temp}
\end{align}
where we conclude the last inequality from $\displaystyle{} \sup_{t> 0} t e^{2\beta d(t)/7} < \infty$ which is a direct consequence of assumption (\ref{rho_assumption1}) as well.

\subsubsection{The case that $\rho$ satisfies assumption {\rm(\ref{rho_assumption2})}}
In this case,
\begin{align*}
d(t) = \int_0^t \frac{6}{\rho(s)^3} ds \le C \int_0^\infty \frac{1}{(1+t)^{3\nu^-}} ds < \infty \qquad \Forall t>0;
\end{align*}
thus $d$ is bounded. 
Since
\begin{align*}
\|\tth(t)\|_{H^k(\bbS^1)} + \rho(t)^{2 - \frac{1}{2\nu^-}} \|\tth(t)\|_{H^{2.5}(\bbS^1)} \le \triplenorm{\tth}_T,
\end{align*}
by interpolation,  (\ref{decay_estimate_process}) implies that
\begin{align*}
& \int_0^t \rho(s)^4 e^{2d(s)} \|R(s)\|^2_{H^{2.5}(\bbS^1)} ds \le C \int_0^t \|\tth(s)\|_{H^{2.5}(\bbS^1)}^{2 + 2 \frac{\rK -5.5}{\rK -2.5}} \|\tth(s)\|^{\frac{6}{\rK -2.5}}_{H^\rK (\bbS^1)} ds \\
&\qquad\qquad + C \int_0^t \rho(s)^2 \rho^\pprime(s)^2 \|\tth(s)\|^{2+2\frac{\rK -3.5}{\rK -2.5}}_{H^{2.5}(\bbS^1)} \|\tth(s)\|^{\frac{2}{\rK -2.5}}_{H^\rK (\bbS^1)} ds \\
&\quad \le C \Big[\int_0^t \Big(\rho(s)^{\frac{1-4\nu^-}{\nu^-}\frac{2\rK-8}{\rK-2.5}} + \rho(s)^{\frac{1-4\nu^-}{\nu^-} \frac{2\rK-6}{\rK-2.5} + 2} \rho^\pprime(s)^2 \Big) ds\Big] \triplenorm{\tth}^4_T .
\end{align*}
Note that by the definition of $\rK$, $\displaystyle{}\rK > \frac{64\nu-21}{16\nu -6}$; thus by (\ref{rho_assumption2}),
\begin{align*}
\rho(s)^{\frac{1-4\nu^-}{\nu^-}\frac{2\rK-8}{\rK-2.5}} + \rho(s)^{\frac{1-4\nu^-}{\nu^-} \frac{2\rK-6}{\rK-2.5} + 2} \rho^\pprime(s)^2 \le C_\sigma (1+t)^{-\sigma}
\end{align*}
for some $\sigma = \sigma(\nu) > 1$. Therefore,
\begin{align*}
\frac{t e^{-2d(t)}}{\rho(t)^4} \int_0^t e^{2d(s)} \|R(s)\|^2_{H^{2.5}(\bbS^1)} ds \le \frac{C_\sigma t}{\rho(t)^4} \int_0^\infty (1+t)^{-\sigma} ds \le \frac{C t}{\rho(t)^4} \,;
\end{align*}
and we then conclude from (\ref{decay_estimate_insu}) that
\begin{align}
\frac{\rho(t)^4}{1+t} \sum_{|k|\ge 2} (1+|k|^5) |\ft{\tth}_k(t)|^2 \le C \Big[\|\tth_0\|^2_{H^{2.5}(\bbS^1)} + \triplenorm{\tth}^4_T\Big]. \label{decay_estimate_temp1}
\end{align}

\subsection{The  mass and the center of mass}\label{sec:center_of_mass}
For ${\bf x} \in \Omega(t)$, we write  ${\bf x} = (x,y)$.  Equation 
 (\ref{defn:rho}) shows that $\int_{\Omega(t)} = \pi \rho(t)^2$; hence, 
 \begin{align*}
\pi \rho(t)^2 = \frac{1}{2} \oint (xdy - ydx) = \frac{1}{2} \int_0^{2\pi} \big(\rho(t)+\tth(\theta,t)\big)^2 d\theta.
\end{align*}
As a consequence,
\begin{align}
\ft{\tth}_0(t) &= - \frac{1}{2\rho} \ft{\tth^2}_0(t) = - \frac{1}{2\rho} \sum_{k\in\bbZ} \ft{\tth}_k \ft{\tth}_{-k} \nonumber\\
&= - \frac{1}{2\rho} \Big[|\ft{\tth}_0|^2 + 2 \ft{\tth}_1 \ft{\tth}_{-1} \Big] - \frac{1}{2\rho} \sum_{k\ne -1,0,1} \ft{\tth}_k \ft{\tth}_{-k}. \label{ft_h_0}
\end{align}
Moreover, by Green's identity and the fact that that $\displaystyle{}\int_{\Gamma(t)} \rH n dS = 0$,
 we see that
\begin{align*}
\frac{d}{dt} \int_{\Omega(t)} (x,y) dA &= \int_{\Gamma(t)} (x,y) (u\cdot n) dS = - \int_{\Gamma(t)} (x,y) \frac{\p p}{\p n}\, dS \\
&= - \int_{\Gamma(t)} \frac{\p (x,y)}{\p n} p \, dS + \int_{\Omega(t)} (x,y) \Delta p\, dx = - \int_{\Gamma(t)} \rH n dS = 0.
\end{align*}
Therefore, the center of mass of $\Omega(t)$ does not change in time; thus
\begin{align*}
x_0 = \int_{\Omega(t)} x dA = \int_0^{2\pi} \int_0^{\rho+\tth} r^2 \cos\theta dr d\theta = \frac{1}{3} \int_0^{2\pi} (\rho+\tth)^3 \cos \theta d\theta
\end{align*}
and
\begin{align*}
y_0 = \int_{\Omega(t)} y dA = \int_0^{2\pi} \int_0^{\rho+\tth} r^2 \sin\theta dr d\theta = \frac{1}{3} \int_0^{2\pi} (\rho+\tth)^3 \sin \theta d\theta.
\end{align*}
Letting $x_0 + iy_0 = r_0 e^{i\theta_0}$, we see that
\begin{align*}
\int_0^{2\pi} \big(\rho(t)+\tth(\theta,t)\big)^3 e^{\pm i \theta} d\theta = 3 r_0 e^{\pm i\theta_0};
\end{align*}
thus
\begin{align*}
\rho^2 \ft{\tth}_{\pm 1} + \rho \ft{\tth^2}_{\pm 1} + \frac{1}{3} \ft{\tth^3}_{\pm 1} = 3 r_0 e^{\pm i\theta_0}.
\end{align*}
In particular, since $\displaystyle{}\ft{fg}_j = \sum_{k\in\bbZ} \ft{f}_k \ft{g}_{j-k}$, we find that
\begin{align}
\ft{\tth}_1 &= \frac{3r_0 e^{i\theta_0}}{\rho^2} - \frac{1}{\rho} \sum_{k\in \bbZ} \ft{\tth}_k \ft{\tth}_{1-k} - \frac{1}{3\rho^2} \sum_{\ell,k\in\bbZ} \ft{\tth}_{k-\ell} \ft{\tth}_\ell \ft{\tth}_{1-k} \nonumber\\
&= \hspace{-1pt}\frac{3r_0 e^{i\theta_0}}{\rho^2} - \hspace{-2pt}\Big[\frac{2 \ft{\tth}_0}{\rho} \hspace{-1pt}+\hspace{-1pt} \frac{\|\tth\|^2_{L^2(\bbS^1)}}{3\rho^2} \hspace{-1pt}+\hspace{-1pt} \frac{2 |\ft{\tth}_0|^2}{3\rho^2} \Big]\ft{\tth}_1 - \big(\frac{1}{\rho} \hspace{-1pt}+\hspace{-1pt} \frac{2\ft{\tth}_0 \hspace{-1pt}+\hspace{-1pt} \ft{\tth}_1 \hspace{-1pt}+\hspace{-1pt} \ft{\tth}_{-1}}{3\rho^2} \big) \hspace{-2pt} \sum_{k\ne 0,1} \hspace{-2pt} \ft{\tth}_k \ft{\tth}_{1-k} \nonumber\\
&\quad - \frac{1}{3\rho^2} \hspace{-2pt}\sum_{|\ell-2|\ge 2} \hspace{-2pt}\ft{\tth}_{2-\ell} \ft{\tth}_\ell \ft{\tth}_{-1} - \frac{1}{3\rho^2} \sum_{k\ne 0,1,2} \sum_{|\ell-k|\ge 2} \ft{\tth}_{k-\ell} \ft{\tth}_\ell \ft{\tth}_{1-k}. \label{ft_h_1}
\end{align}
By (\ref{decay_estimate_temp}) and the Schwartz inequality,
\begin{align}
\sum_{k\ne 0,\pm 1} \ft{\tth}_k \ft{\tth}_{\pm 1-k} &\le \frac{C}{\D(t)} \Big[\|\tth_0\|_{H^{2.5}(\bbS^1)} + \triplenorm{\tth}_T^2 \Big] \sum_{k\in\bbZ} \frac{|\ft{\tth}_k|}{1+|k|^{2.5}} \nonumber\\
&\le \frac{C}{\D(t)} \Big[\|\tth_0\|_{H^{2.5}(\bbS^1)} + \triplenorm{\tth}_T^2 \Big] \|\tth\|_{L^2(\bbS^1)}. \label{ft_h_1_est1}
\end{align}
On the other hand, since $|\ft{\tth}_j| \le \|\tth\|_{L^2(\bbS^1)}$,
\begin{align}
\sum_{|\ell-2|\ge 2} \ft{\tth}_{2-\ell} \ft{\tth}_\ell \ft{\tth}_{-1} &\le \frac{C}{\D(t)} \Big[\|\tth_0\|_{H^{2.5}(\bbS^1)} + \triplenorm{\tth}_T^2 \Big] \sum_{|\ell-2|\ge 2} \frac{|\ft{\tth}_\ell| |\ft{\tth}_{-1}|}{1 + |\ell-2|^{2.5}} \nonumber\\
&\le \frac{C}{\D(t)} \Big[\|\tth_0\|_{H^{2.5}(\bbS^1)} + \triplenorm{\tth}_T^2 \Big] \|\tth\|^2_{L^2(\bbS^1)}, \label{ft_h_1_est2}
\end{align}
and similarly,
\begin{align}
&\sum_{k\ne 0,1,2} \sum_{|\ell-k|\ge 2} \ft{\tth}_{k-\ell} \ft{\tth}_\ell \ft{\tth}_{1-k} \nonumber\\
&\quad \le \frac{C}{\D(t)^2} \Big[\|\tth_0\|^2_{H^{2.5}(\bbS^1)} + \triplenorm{\tth}_T^4 \Big] \sum_{k\ne 0,1,2} \sum_{|\ell-k|\ge 2} \frac{|\ft{\tth}_\ell|}{(1+|k-\ell|^{2.5})(1+|k|^{2.5})} \nonumber\\
&\quad \le \frac{C}{\D(t)^2} \Big[\|\tth_0\|^2_{H^{2.5}(\bbS^1)} + \triplenorm{\tth}_T^4 \Big] \|\tth\|_{L^2(\bbS^1)}. \label{ft_h_1_est3}
\end{align}
Moreover, by assumption (\ref{assumption1}), $|\ft{\tth}_0| \le 2\pi \sigma \ll 1$ and $\|\tth\|_{L^2(\bbS^1)} \le \sqrt{2\pi} \sigma \ll 1$. As a consequence, (\ref{ft_h_1}) together with (\ref{ft_h_1_est1}), (\ref{ft_h_1_est2}) and (\ref{ft_h_1_est3}) implies that
\begin{align}
|\ft{\tth}_1| \hspace{-1pt}\le\hspace{-1pt} \frac{C}{\D(t)} \Big[\|\tth_0\|_{H^{2.5}(\bbS^1)} \hspace{-2pt}+ \triplenorm{\tth}^2_T \hspace{-2pt}+ \triplenorm{\tth}^4_T \Big]. \label{ft_h_1_decay_estimate}
\end{align}
A similar argument also suggests that $|\ft{\tth}_{-1}|$ shares the same upper bound as $|\ft{\tth}_1|$. Therefore, once again using the inequality 
$|\ft{\tth}_0|\le 2\pi \sigma \ll 1$, (\ref{ft_h_0}) together with (\ref{ft_h_1_est1}) implies that
\begin{align}
|\ft{\tth}_0| &\le C \rho^{-1} \Big[|\ft{\tth}_1 \ft{\tth}_{-1}| + \frac{1}{\D(t)} \Big(\|\tth_0\|_{H^{2.5}(\bbS^1)} + \triplenorm{\tth}_T^2 \Big) \|\tth\|_{L^2(\bbS^1)} \Big] \nonumber\\
&\le \frac{C}{\D(t)} \Big[\|\tth_0\|^2_{H^{2.5}(\bbS^1)} \hspace{-2pt}+ \triplenorm{\tth}_T^2 \hspace{-2pt}+ \triplenorm{\tth}_T^8 \Big]. \label{ft_h_0_decay_estimate}
\end{align}
Combining (\ref{decay_estimate_temp}) (or (\ref{decay_estimate_temp1})), (\ref{ft_h_1_decay_estimate}) and (\ref{ft_h_0_decay_estimate}), we conclude that
\begin{align}
&\|\tth\|_{H^{2.5}(\bbS^1)} = \Big[\sum_{k\in\bbZ} (1+|k|^5) |\ft{\tth}_k|^2\Big]^\frac{1}{2} \hspace{-2pt} \le \frac{C}{\D(t)} \Big[\|\tth_0\|^2_{H^{2.5}(\bbS^1)}\hspace{-2pt} + \triplenorm{\tth}_T^2 \hspace{-2pt}+ \triplenorm{\tth}_T^8 \Big]. \label{decay_estimate}
\end{align}
Furthermore, since
\begin{align*}
\int_0^\infty \frac{\rho(t)^{-3} + \rho(t) \rho^\pprime(t)^2}{\D(t)^2} \,dt < \infty
\end{align*}
for $\D$ in both the slow and fast injection cases, (\ref{elliptic_est1}) implies that
\begin{equation}\label{elliptic_est}
\begin{array}{l}
\displaystyle{} \int_0^\infty \hspace{-2pt} \rho(t)^{-3} \|\tth(t)\|^2_{H^{\rK+1.5}(\bbS^1)} dt \vspace{.2cm}\\
\displaystyle{} \qquad\quad \le C \int_0^\infty \hspace{-2pt} \rho(t)^3 \|v(t)\|^2_{H^{k-1}(B_1)} dt + C \Big[ \|\tth_0\|^2_{H^{2.5}(\bbS^1)} \hspace{-2pt} + \triplenorm{\tth}_T^4 \hspace{-2pt}+ \triplenorm{\tth}_T^{16} \Big].
\end{array}
\end{equation}


\section{Energy estimates in the space of higher regularity}\label{sec:high_norm_est}
 Let $\displaystyle{}\bp = \rT \cdot \nabla = \frac{\p}{\p \theta}$ denote the tangential derivative. 
 Tangentially differentiating (\ref{HS_ALE1}a) $\ell$-times ($\ell = 0,1,2,3,4,5$) and then testing the resulting equation against $\rJ \bp^\ell v$, we find that
\begin{align*}
& \int_{B_1} \rJ |\bp^\ell \ttv|^2 dx + \int_{B_1} \rJ \rA^j_i \bp^\ell \ttq,_j \bp^\ell \ttv^i dx = - \sum_{k=0}^{\ell-1} {{\ell}\choose{k}} \int_{B_1} \rJ \bp^{\ell-k} \rA^j_i \bp^k \ttq,_j \bp^\ell \ttv^i dx.
\end{align*}
Writing the second integral on the left-hand side as
\begin{align*}
& \int_{B_1} \rJ \rA^j_i \bp^\ell \ttq,_j \bp^\ell \ttv^i dx = \int_{B_1} \rJ \rA^j_i (\bp^\ell \ttq),_j \bp^\ell \ttv^i dx + \int_{B_1} \rJ \rA^j_i \big[\bp^\ell \ttq,_j - (\bp^\ell \ttq),_j\big] \bp^\ell \ttv^i dx,
\end{align*}
then integrating by parts in $x_j$ leads to 
\begin{align}
& \int_{B_1} \rJ |\bp^\ell \ttv|^2 dx + \int_{\bbS^1} \rJ \rA^j_i \bp^\ell \ttq \bp^\ell \ttv^i \rN_j dS = - \sum_{k=0}^{\ell-1} {{\ell}\choose{k}} \int_{B_1} \rJ \bp^{\ell-k} \rA^j_i \bp^k \ttq,_j \bp^\ell \ttv^i dx \nonumber\\
&\qquad\qquad + \int_{B_1} \rJ \rA^j_i \Big[\big[(\bp^\ell \ttq),_j - \bp^\ell \ttq,_j\big] \bp^\ell \ttv^i + \bp^\ell \ttq (\bp^\ell \ttv^i),_j \Big] dx. \label{energy_id1}
\end{align}
Using the  Kronecker delta symbol $\delta_{0\ell}$, which vanishes for all $\ell \ne 0$, by (\ref{A_est2}), (\ref{J-1_est}) and (\ref{ttq_in_terms_of_ttv1}) together with the continuous embedding $H^{\frac{p}{p-2}}(B_1) \contsubset L^p(B_1)$ we find that
\begin{align*}
& \Big|\int_{B_1} \rJ \rA^j_i \Big[\big[\bp^\ell \ttq,_j - (\bp^\ell \ttq),_j\big] \bp^\ell \ttv^i + \bp^\ell \ttq \big[(\bp^\ell \ttv^i),_j - \bp^\ell \ttv^i,_j\Big] dx\Big| \\
&\qquad\qquad\quad \le C \rho^2 (1-\delta_{0\ell}) \|\ttv\|_{H^{\ell-1}(B_1)} \|\ttv\|_{H^\ell(B_1)},
\end{align*}
as well as for $0\le k\le \ell-2$,
\begin{align*}
& \Big|\int_{B_1} \rJ \bp^{\ell-k} \rA^j_i \bp^k \ttq,_j \bp^\ell \ttv^i dx\Big| \le C \rho^2 \|D\rA\|_{W^{\ell-k-1,4}(B_1)} \|\nabla \ttq\|_{W^{k,4}(B_1)} \|\ttv\|_{H^\ell(B_1)} \\
&\qquad\qquad\quad \le C \rho \|\tth\|_{H^{\ell-k+1}(\bbS^1)} \|\ttv\|_{H^{k+0.5}(B_1)} \|\ttv\|_{H^\ell(B_1)}
\end{align*}
and for $k = \ell-1$,
\begin{align*}
& \Big|\int_{B_1} \rJ \bp^{\ell-k} \rA^j_i \bp^k \ttq,_j \bp^\ell \ttv^i dx\Big| \le C \rho^3 \|D\rA\|_{L^\infty(B_1)} \|\ttv\|_{H^{\ell-1}(B_1)} \|\ttv\|_{H^\ell(B_1)} \\
&\qquad\qquad\qquad \le C \rho \|\tth\|_{H^3(\bbS^1)} \|\ttv\|_{H^{\ell-1}(B_1)} \|\ttv\|_{H^\ell(B_1)}.
\end{align*}
Therefore, (\ref{ttvq_est}) suggests that
\begin{align}
&\int_{B_1} \hspace{-2pt} \rJ |\bp^\ell \ttv|^2 dx + \int_{\bbS^1} \rJ \rA^j_i \bp^\ell \ttq \bp^\ell \ttv^i \rN_j dS \nonumber\\
&\qquad\quad \le \int_{B_1} \hspace{-3pt} \rJ \rA^j_i \bp^\ell \ttq \bp^\ell \ttv^i,_j dx + C (1-\delta_{0\ell}) \rho^2 \|\ttv\|_{H^{\ell-1}(B_1)} \|\ttv\|_{H^\ell(B_1)} \label{energy_ineq1}\\
&\qquad\qquad + C (1-\delta_{0\ell}) \|\tth\|_{H^{\ell+1}(\bbS^1)} \|\ttv\|_{H^\ell(B_1)} \|\ttv\|_{H^1(B_1)} 
. \nonumber
\end{align}
By (\ref{HS_ALE1}b), $\displaystyle{}\rA^j_i \bp^\ell \ttv^i,_j = - \sum_{k=0}^{\ell-1} {{\ell}\choose{k}} \bp^{\ell-k} \rA^j_i \bp^k \ttv^i,_j $; thus for $\ell\ge 1$,
\begin{align*}
&\int_{B_1} \rJ \bp^\ell \ttq \rA^j_i \bp^\ell \ttv^i,_j dx \le C \rho^2 (1-\delta_{0\ell}) \Big[ \|\nabla \ttq\|_{H^{\ell-1}(\hspace{-0.5pt}B_1\hspace{-0.5pt})} \|D \rA\|_{L^\infty(\hspace{-0.5pt}B_1\hspace{-0.5pt})} \|\ttv\|_{H^\ell(\hspace{-0.5pt}B_1\hspace{-0.5pt})} \\
&\qquad\qquad\quad + \|\nabla\ttq\|_{H^{\ell-0.5}(\hspace{-0.5pt}B_1\hspace{-0.5pt})} \sum_{k=0}^{\ell-2} \|D \rA\|_{H^{\ell-k-0.5}(\hspace{-0.5pt}B_1\hspace{-0.5pt})} \|\ttv\|_{H^{k+1}(\hspace{-0.5pt}B_1\hspace{-0.5pt})} \Big] \\
&\ \le \hspace{-1pt} C \rho (1\hspace{-1pt}-\hspace{-1pt}\delta_{0\ell}) \Big[ \|\tth\|_{H^{\ell+1}(\bbS^1)} \|\ttv\|_{H^\ell(\hspace{-0.5pt}B_1\hspace{-0.5pt})} \|\ttv\|_{H^1(\hspace{-0.5pt}B_1\hspace{-0.5pt})} \hspace{-3pt} + \|\tth\|_{H^3(\bbS^1)} \|\ttv\|_{H^{\ell-1}(\hspace{-0.5pt}B_1\hspace{-0.5pt})} \|\ttv\|_{H^\ell(\hspace{-0.5pt}B_1\hspace{-0.5pt})} \Big].
\end{align*}

We now focus on the second integral of the left-hand side of (\ref{energy_ineq1}). By identity (\ref{JAtN}) and the boundary condition (\ref{HS_ALE1}c),
\begin{align*}
& \int_{\bbS^1} \rJ \rA^j_i \bp^\ell \ttq \bp^\ell \ttv^i \rN_j dS = \int_{\bbS^1} (\rho+\tth) (\bp^\ell \ttv \cdot \N) \bp^\ell \ttq dS \\
&\qquad = \int_{\bbS^1} (\rho + \tth) \Big[\bp^\ell (\ttv \cdot \N) - \sum_{k=0}^{\ell-1} {{\ell}\choose{k}} \bp^k \ttv \cdot \bp^{\ell-k} \N \Big] \bp^\ell \ttq\, dS \\
&\qquad = \int_{\bbS^1} (\rho + \tth) \bp^\ell (\ttv \cdot \N) \bp^\ell \Big[-\ttJ_\tth^{-3} \big[(\rho + \tth) \tth_{\theta\theta}\big] + \rho^\pprime \tth \Big] dS \\
&\qquad\quad + \int_{\bbS^1} (\rho + \tth) \bp^\ell (\ttv \cdot \N) \bp^\ell \Big[\ttJ_\tth^{-3} \big(\ttJ_\tth^2 + \tth_\theta^2 \big) - \rho^{-1} + \bar{\bar{\ttq}} \Big] dS \\
&\qquad\quad - \sum_{k=0}^{\ell-1} {{\ell}\choose{k}} \int_{\bbS^1} (\rho+\tth) \bp^k \ttv \cdot \bp^{\ell-k}\N \bp^\ell \ttq dS .
\end{align*}
By employing the $H^{0.5}(\bbS^1)$-$H^{-0.5}(\bbS^1)$ duality pairing,
\begin{align*}
& \Big|\sum_{k=0}^{\ell-1} \int_{\bbS^1} (\rho + \tth) \bp^k \ttv \hspace{-1pt}\cdot\hspace{-1pt} \bp^{\ell-k} \N \bp^\ell \ttq dS \Big| \hspace{-2pt}\le\hspace{-1pt} C (1-\delta_{0\ell}) \rho \Big[\|\ttv\|_{L^4(\bbS^1)} \|\bp^\ell \N\|_{L^2(\bbS^1)} \|\bp^\ell \ttq\|_{L^4(\bbS^1)} \\
&\qquad\qquad + (1-\delta_{0\ell}-\delta_{1\ell}) \|\bp^{\ell-1} \ttv\|_{H^{-0.5}(\bbS^1)} \|\bp^\ell \ttq\|_{H^{0.5}(\bbS^1)} \|\bp \N\|_{H^1(\bbS^1)} \\
&\qquad\qquad + (1-\delta_{0\ell}-\delta_{1\ell}-\delta_{2\ell}) \sum_{k=1}^{\ell-2} \|\bp^k \ttv\|_{H^{-0.5}(\bbS^1)} \|\bp^\ell \ttq\|_{H^{0.5}(\bbS^1)} \|\bp^{\ell-k}\N\|_{H^1(\bbS^1)} \Big] \\
&\qquad\quad \le C \rho^2 \Big[ \|\ttv\|_{H^\ell(B_1)} \|\ttv\|_{H^1(B_1)} \big[1 + \rho^{-1} \|\tth\|_{H^{\ell+1}(\bbS^1)}\big] \\
&\qquad\qquad\qquad + \|\ttv\|_{H^{\ell-1}(B_1)} \|\ttv\|_{H^\ell(B_1)} \big[1+ \rho^{-1} \|\tth\|_{H^4(\bbS^1)} \big]\Big]
\end{align*}
and by (\ref{qbb_est}),
\begin{align*}
& \Big|\int_{\bbS^1} (\rho + \tth) \bp^\ell (\ttv \cdot \N) \bp^\ell \bar{\bar{\ttq}} dS \Big| \le C \rho \|\ttv\cdot\N\|_{H^{\ell-0.5}(\bbS^1)} \|\bar{\bar{\ttq}}\|_{H^{\ell+0.5}(\bbS^1)} \\
&\qquad\qquad \le C \rho^\pprime \|\ttv\cdot \N\|_{H^{\ell-0.5}(\bbS^1)} \|\tth\|_{H^2(\bbS^1)} \|\tth\|_{H^{\ell+0.5}(\bbS^1)}.
\end{align*}
Moreover, since
$$\displaystyle{}\ttJ_\tth^{-1} - \rho^{-1} = - \frac{2\rho\tth + \tth^2 + \tth^2_\theta}{\ttJ_\tth \rho (\rho + \ttJ_\tth)}$$ and $$\displaystyle{}\frac{2(\rho+\tth)}{\ttJ_\tth(\rho + \ttJ_\tth)} - \frac{1}{\rho} = - \Big[ \frac{\tth(2\rho + \tth) + \tth_\theta^2}{\ttJ_\tth (\rho + \ttJ_\tth)^2} + \frac{\tth^2 + \tth^2_\theta}{\ttJ_\tth \rho (\rho + \ttJ_\tth)}\Big]\,,$$
we find that
\begin{align*}
& - \int_{\bbS^1} (\rho + \tth) \bp^\ell (\ttv\cdot \N) \bp^\ell \big[\ttJ_\tth^{-1} - \rho^{-1}\big] dS = \int_{\bbS^1} (\rho + \tth)\bp^\ell (\ttv \cdot \N) \bp^\ell \Big[\frac{2 \tth}{\ttJ_\tth (\rho + \ttJ_\tth)}\Big] \, dS \\
&\qquad  + \int_{\bbS^1} (\rho + \tth) \bp^\ell (\ttv\cdot \N) \bp^\ell \Big[\frac{\tth^2 + \tth_\theta^2}{\ttJ_\tth \rho (\rho + \ttJ_\tth)} \Big] dS \\
&\quad \le \int_{\bbS^1} \frac{2 (\rho+\tth)}{\ttJ_\tth (\rho + \ttJ_\tth)} \bp^\ell (\ttv\cdot \N) \bp^\ell \tth dS + C \rho^{-2} \|\tth\|_{H^{2.5}(\bbS^1)} \|\tth\|_{H^{\ell+1.5}(\bbS^1)} \|\ttv\|_{H^\ell(B_1)} \\
&\quad \le \frac{1}{\rho} \int_{\bbS^1} \bp^\ell (\ttv\cdot \N) \bp^\ell \tth\, dS + C \rho^{-2} \|\tth\|_{H^{2.5}(\bbS^1)} \|\tth\|_{H^{\ell+1.5}(\bbS^1)} \|\ttv\cdot \N\|_{H^{\ell-0.5}(\bbS^1)} \\
&\qquad + C \rho^{-2} \|\tth\|_{H^{2.5}(\bbS^1)} \|\tth\|_{H^{\ell+1.5}(\bbS^1)} \|\ttv\|_{H^\ell(B_1)} .
\end{align*}
Finally, because of the identity
\begin{align*}
\frac{(\rho+\tth)^2}{\ttJ_\tth^3} - \frac{1}{\rho} = - \frac{(\rho+\tth)^4 (2\rho + \tth)\tth + 3 (\rho+\tth)^4 \tth_\theta^2 + 3 (\rho+\tth)^2 \tth_\theta^4 + \tth_\theta^6}{\ttJ_\tth^3 \rho \big[\ttJ_\tth^3 + \rho (\rho+\tth)^2\big]}
\end{align*}
by the evolution equation (\ref{HS_ALE1}d) we obtain that
\begin{align*}
& \int_{\bbS^1} (\rho + \tth) \bp^\ell (\ttv \cdot \N) \bp^\ell \Big[-\ttJ_\tth^{-3} (\rho+\tth) \tth_{\theta\theta} + \rho^\pprime \tth \Big] dS \\
&\quad = - \int_{\bbS^1} \ttJ_\tth^{-3} (\rho+\tth)^2 \bp^\ell (\ttv\cdot \N) \bp^{\ell+2} \tth \, dS + \int_{\bbS^1} (\rho + \tth) \rho^\pprime \bp^\ell (\ttv\cdot \N) \bp^\ell \tth\, dS \\
&\qquad - \sum_{k=0}^{\ell-1} {{\ell}\choose{k}} \int_{\bbS^1} (\rho+\tth) \bp^\ell(\ttv\cdot \N) \bp^{\ell-k} \big[\ttJ_\tth^{-3}(\rho+h)\big] \bp^{k+1} \tth \, dS \\
&\quad \ge \frac{1}{2 \rho} \frac{d}{dt} \|\bp^{\ell+1} \tth\|^2_{L^2(\bbS^1)} + \int_{\bbS^1} (\rho+\tth) \rho^\pprime \bp^\ell (\ttv\cdot \N) \bp^\ell \tth\, dS \\
&\qquad - C \|\tth\|_{H^2(\bbS^1)} \|\ttv\cdot \N\|_{H^{\ell-0.5}(\bbS^1)} \Big[\rho^{-2} \|\tth\|_{H^{\ell+2.5}(\bbS^1)} + \rho^\pprime \|\tth\|_{H^{\ell+0.5}(\bbS^1)} \Big] \\
&\qquad - C \rho^{-1} \|\ttv\cdot \N\|_{H^{\ell-0.5}(\bbS^1)} \Big[\|\tth\|_{H^2(\bbS^1)} \|\tth\|_{H^{\ell+0.5}(\bbS^1)} + \|\tth\|^2_{H^2(\bbS^1)} \|\tth\|_{H^{\ell+1.5}(\bbS^1)} \Big].
\end{align*}
Since $\div (\rJ \rA^\rT \tv) = 0$, by the normal trace estimate,  we find that
\begin{align*}
& \|\ttv\cdot \N\|_{H^{-0.5}(\bbS^1)} = \|(\rho+\tth)^{-1} \rJ \rA^\rT \ttv \cdot \rN\|_{H^{-0.5}(\bbS^1)} \\
&\quad \le C \Big[\|(\rho+\tth)^{-1} \rJ \rA^\rT \ttv\|_{L^2(B_1)} + \|\nabla (\rho+\tth)^{-1} \rJ \rA^\rT \ttv\|_{L^2(B_1)} \Big] \le C \|\ttv\|_{L^2(B_1)}
\end{align*}
and similarly, for $0\le \ell \le \rK-1$,
\begin{align*}
& \|\ttv\cdot \N\|_{H^{\ell-0.5}(\bbS^1)} \le C \big(1+\rho^{-1} \triplenorm{\tth}_T\big) \|\ttv\|_{H^\ell(B_1)}.
\end{align*}
Therefore, (\ref{energy_ineq1}) implies that
\begin{align*}
& \int_{B_1} \hspace{-2pt} \rho \rJ |\bp^\ell \ttv|^2 dx \hspace{-1pt}+\hspace{-1pt} \frac{1}{2} \frac{d}{dt} \|\bp^{\ell+1}\tth\|^2_{L^2(\bbS^1)} + \int_{\bbS^1} \rho \rho^\pprime (\rho+\tth) \bp^\ell (\ttv\cdot \N) \bp^\ell \tth dS \\
& \qquad \le \int_{\bbS^1} \bp^\ell (\ttv\cdot \N) \bp^\ell \tth dS + C (1-\delta_{0\ell}) \rho^3 \|\ttv\|_{H^{\ell-1}(B_1)} \|\ttv\|_{H^\ell(B_1)} \\
& \qquad\quad + C \Big[\rho^{-1} \|\tth\|_{H^{\ell+2.5}(\bbS^1)} + \rho \rho^\pprime \|\tth\|_{H^{\ell+0.5}(\bbS^1)} \Big] \|\tth\|_{H^{2.5}(\bbS^1)} \|\ttv\|_{H^\ell(B_1)}.
\end{align*}
By Young's inequality, the inequality above further implies that
\begin{align}
&\rho^3 \|\bp^\ell \ttv\|^2_{L^2(B_1)} + \frac{d}{dt} \|\bp^{\ell+1} \tth\|^2_{L^2(\bbS^1)} + 2 \int_{\bbS^1} \rho \rho^\pprime (\rho + \tth) \bp^\ell (\ttv\cdot \N) \bp^\ell \tth dS \nonumber\\
&\qquad \le 2 \int_{\bbS^1} \bp^\ell (\ttv\cdot \N) \bp^\ell \tth dS + C (1-\delta_{0\ell}) \rho^3 \|\ttv\|^2_{H^{\ell-1}(B_1)} \hspace{-2pt} + \delta \rho^3 \|\ttv\|^2_{H^\ell(B_1)} \label{energy_est_temp} \\
&\qquad\quad + C_\delta \rho^{-1} \rho^{\pprime\hspace{1pt}2} \|\tth\|^2_{H^{2.5}(\bbS^1)} \|\tth\|^2_{H^{\ell+1}(\bbS^1)} + C_\delta \rho^{-5} \|\tth\|^2_{H^{2.5}(\bbS^1)} \|\tth\|^2_{H^{\ell+2.5}(\bbS^1)}. \nonumber
\end{align}

\subsection{The case $\ell=0$}
Using the $H^{0.5}(\bbS^1)$-$H^{-0.5}(\bbS^1)$ duality pairing,
\begin{align*}
& \Big|\int_{\bbS^1} (\rho \rho^\pprime (\rho+\tth) (\ttv\cdot \N) \tth dS \Big| \hspace{-2pt}+ \Big|\int_{\bbS^1} (\ttv\cdot \N) \tth dS\Big| \hspace{-1pt}\le C (\rho^2 \rho^\pprime + 1) \|\ttv\|_{L^2(B_1)} \|\tth\|_{H^{0.5}(\bbS^1)} \\
&\qquad\quad \le C_\delta \big(\rho \rho^{\pprime\hspace{1pt}2} + \rho^{-3}\big) \|\tth\|^2_{H^{2.5}(\bbS^1)} + \delta \rho^3 \|\ttv\|^2_{L^2(B_1)}.
\end{align*}
Since
$$
\int_0^\infty \Big[\frac{1}{\rho(s)^3 \D(s)^2} + \frac{\rho^\pprime(s)^2}{\rho(s) \D(s)^4}\Big] ds < \infty\,,
$$
choosing $\delta>0$ small enough and integrating in time of (\ref{energy_est_temp}) over the time interval $(0,t)$, by (\ref{decay_estimate}) we find that
\begin{align*}
& \|\bp \tth(t)\|^2_{L^2(\bbS^1)} + \int_0^t \rho(s)^2 \|\ttv(s)\|^2_{L^2(B_1)} ds \\
&\qquad \le C \int_0^t \rho(s) \rho^\pprime(s)^2 \|\tth(s)\|^2_{H^{2.5}(\bbS^1)} ds + C \Big[\|\tth_0\|^2_{H^{2.5}(\bbS^1)} \hspace{-2pt} + \triplenorm{\tth}_T^4 \hspace{-2pt}+ \triplenorm{\tth}_T^{16} \Big].
\end{align*}
Define
\begin{align}
\Nhh \equiv \|\tth_0\|^2_{H^\rK(\bbS^1)} \hspace{-2pt} + \triplenorm{\tth}_T^4 \P\big(\triplenorm{\tth}^2_T\big)\label{defn:Nh0hT}
\end{align}
for some polynomial function $\P$. If (\ref{rho_assumption1}) is satisfied, due to the exponential decay of $\|\tth\|_{H^{2.5}(\bbS^)}$,  it is easy to see that
\begin{align*}
\int_0^t \hspace{-2pt} \rho(s) \rho^\pprime(s)^2 \|\tth(s)\|^2_{H^{2.5}(\bbS^1)} ds \le C \Nhh.
\end{align*}
Now suppose that (\ref{rho_assumption2}) is satisfied. Then (\ref{rho_grows_algebraically}) and (\ref{decay_estimate}) imply that
\begin{align*}
& \int_0^t \rho(s) \rho^\pprime(s)^2 \|\tth(s)\|^2_{H^{2.5}(\bbS^1)} ds \le C \Big[\int_0^\infty \frac{\rho(s)^{1/\nu^-}}{(1+s)^2 \rho(s)} ds\Big] \Nhh \\
&\quad \le C \Big[\int_0^\infty \hspace{-2pt} \frac{(1+t)^{\frac{\nu^+}{\nu^-}}}{(1+t)^{2+\nu^-}}\, ds\Big] \Nhh \le C \Nhh;
\end{align*}
thus in either case,
\begin{equation}\label{L2_energy_est}
\|\bp \tth(t)\|^2_{L^2(\bbS^1)} + \int_0^t \rho(s)^3 \|\ttv(s)\|^2_{L^2(B_1)} ds \le C \Nhh.
\end{equation}

\subsection{The case $1\le \ell \le \rK-1$}
Define $\displaystyle{} \widetilde{\tth} = \tth + \frac{\tth^2}{2\rho}$\,. Then
\begin{align*}
\frac{\rho + \tth}{\rho}\, \ttv\cdot \N = \frac{\big[(\rho + \tth)^2\big]_t}{2\rho} = \frac{\big[\rho^2 + 2\rho\widetilde{\tth} \big]_t}{2\rho} = \widetilde{\tth}_t + \frac{\rho^\pprime}{\rho} \widetilde{\tth} + \rho^\pprime.
\end{align*}
Therefore, since $\|\ttv\|_{H^{1.5}(B_1)} \le C \rho^{-1} \Big[\|\tth\|_{H^4(\bbS^1)} + \rho^\pprime \|\tth\|_{H^2(\bbS^1)}\Big]$ by (\ref{ttvq_est}), we find that
\begin{align}
& \int_{\bbS^1} \rho \rho^\pprime (\rho + \tth) \bp^\ell (\ttv\cdot \N) \bp^\ell \tth dS \nonumber\\
&\qquad = \int_{\bbS^1} \rho \rho^\pprime (\rho+\tth) \bp^\ell (\ttv\cdot \N) \bp^\ell \widetilde{\tth} dS - \frac{1}{2} \int_{\bbS^1} \rho^\pprime (\rho+\tth) \bp^\ell (\ttv\cdot \N) \bp^\ell (\tth^2) dS \nonumber\\
&\qquad \ge \int_{\bbS^1} \rho \rho^\pprime \bp^\ell \big[(\rho+\tth)(\ttv\cdot \N)\big]\bp^\ell \widetilde{\tth} dS - \sum_{k=0}^{\ell-1} {{\ell}\choose{k}} \int_{\bbS^1} \rho \rho^\pprime \bp^{\ell-k} \tth \bp^k (\ttv\cdot \N) \bp^\ell \widetilde{\tth} dS \nonumber\\
&\qquad\quad - C \rho \rho^\pprime \|\tth\|_{H^2(\bbS^1)} \|\tth\|_{H^{\ell+1}(\bbS^1)} \|\ttv\|_{H^\ell(B_1)} \nonumber \\
&\qquad \ge \int_{\bbS^1} \rho^2 \rho^\pprime \bp^\ell \big(\widetilde{\tth}_t + \frac{\rho^\pprime}{\rho} \widetilde{\tth}\big) \bp^\ell \widetilde{\tth} dS - C_\delta \rho^{-1} \rho^{\pprime\hspace{1pt}2} \|\tth\|^2_{H^{2.5}(\bbS^1)} \|\tth\|^2_{H^{\ell+1}(\bbS^1)} \label{lower_order_energy_term}\\
&\qquad\quad - \delta \rho^3 \|\ttv\|^2_{H^\ell(B_1)} - C \rho^\pprime \Big[\|\tth\|_{H^4(\bbS^1)} + \rho^\pprime \|\tth\|_{H^2(\bbS^1)} \Big] \|\bp^\ell \tth\|^2_{L^2(\bbS^1)}. \nonumber
\end{align}
Since $\rho \sqrt{\rho^\pprime} \|\tth\|_{H^5(\bbS^1)} \le \triplenorm{\tth}_T$,
\begin{align}
& \int_0^\infty \Big[\rho^\pprime \|\tth(s)\|_{H^4(\bbS^1)} + \rho^\pprime(s)^2 \|\tth(s)\|_{H^2(\bbS^1)} \Big] \|\bp^\ell \tth\|^2_{L^2(\bbS^1)} ds \nonumber\\
&\qquad \le \int_0^\infty \Big[\rho^\pprime \|\tth(s)\|_{H^4(\bbS^1)} + \rho^\pprime(s)^2 \|\tth(s)\|_{H^2(\bbS^1)} \Big] \|\bp^\ell \widetilde{\tth}\|^2_{L^2(\bbS^1)} ds \nonumber\\
&\qquad\quad + C \int_0^\infty \frac{1}{\rho^2} \Big[\rho^\pprime \|\tth(s)\|_{H^4(\bbS^1)} + \rho^\pprime(s)^2 \|\tth(s)\|_{H^2(\bbS^1)} \Big] \|\bp^\ell (\tth^2)\|^2_{L^2(\bbS^1)} ds \nonumber\\
&\qquad \le C \Big[\int_0^\infty \hspace{-2pt}\frac{\|\tth(s)\|_{H^4(\bbS^1)} + \rho^\pprime(s) \|\tth(s)\|_{H^2(\bbS^1)} }{\rho(s)^2} ds \Big] \triplenorm{\tth}^2_T \label{worst_error}\\
&\qquad\quad + C \Nhh \nonumber
\end{align}
By interpolation,
\begin{align*}
\|\tth(t)\|_{H^4(\bbS^1)} \le C \|\tth\|^{\frac{\rK-4}{\rK-2.5}}_{H^{2.5}(\bbS^1)} \|\tth\|^{\frac{1.5}{\rK-2.5}}_{H^\rK(\bbS^1)} \le C \D(t)^{-(\frac{\rK-4}{\rK-2.5})} \triplenorm{\tth}_T.
\end{align*}
If (\ref{rho_assumption1}) is satisfied, due to the exponential decay it is easy to see that
\begin{align}
\int_0^\infty \hspace{-2pt}\frac{\|\tth(s)\|_{H^4(\bbS^1)} + \rho^\pprime(s) \|\tth(s)\|_{H^2(\bbS^1)} }{\rho(s)^2} ds \le C \triplenorm{\tth}_T. \label{worst_error_est}
\end{align}
Suppose that (\ref{rho_assumption2}) is satisfied. Then
\begin{align*}
&\int_0^\infty \hspace{-2pt}\frac{\|\tth(s)\|_{H^4(\bbS^1)} + \rho^\pprime(s) \|\tth(s)\|_{H^2(\bbS^1)} }{\rho(s)^2} ds \\
&\qquad \le C \Big[\int_0^\infty \hspace{-4pt}\Big(\frac{1}{\rho(s)^{2 + (2-\frac{1}{2\nu^-})(\frac{\rK-4}{\rK-2.5})}} + \frac{1}{(1+s) \rho(s)^{3-\frac{1}{2\nu^-}}}\Big)ds\Big] \triplenorm{\tth}_T \\
&\qquad \le C \Big[\int_0^\infty \hspace{-2pt}\frac{1}{(1+s)^{2\nu^- + (2\nu^--\frac{1}{2})(\frac{\rK-4}{\rK-2.5})}} ds\Big] \triplenorm{\tth}_T + C \triplenorm{\tth}_T.
\end{align*}
By the definition of $\rK$, the exponent of the integrand $\displaystyle{}2\nu^- + \big(2\nu^--\frac{1}{2}\big)\big(\frac{\rK-4}{\rK-2.5}\big) > 1$ if $\nu^- > 3/8$, so (\ref{worst_error_est}) is still valid if (\ref{rho_assumption2}) holds. Therefore, (\ref{worst_error}) implies that
\begin{align*}
& \int_0^\infty \Big[\rho^\pprime \|\tth(s)\|_{H^4(\bbS^1)} + \rho^\pprime(s)^2 \|\tth(s)\|_{H^2(\bbS^1)} \Big] \|\bp^\ell \tth\|^2_{L^2(\bbS^1)} ds \\
&\qquad\quad \le C \triplenorm{\tth}^3_T + C \Nhh \le C_{\delta_1} \Nhh + \delta_1 \triplenorm{\tth}^2_T.
\end{align*}
Since
\begin{align*}
\int_{\bbS^1} \rho^2 \rho^\pprime \bp^\ell \big(\widetilde{\tth}_t + \frac{\rho^\pprime}{\rho} \widetilde{\tth}\big) \bp^\ell \widetilde{\tth} dS = \frac{1}{2} \frac{d}{dt} \Big[\rho^2 \rho^\pprime \|\bp^\ell \widetilde{\tth}\|^2_{L^2(\bbS^1)}\Big] - \frac{1}{2} \rho^2 \rho^{\pprime\prime} \|\bp^\ell \widetilde{\tth}\|^2_{L^2(\bbS^1)},
\end{align*}
the combination of (\ref{energy_est_temp}), (\ref{lower_order_energy_term}), (\ref{worst_error}) and (\ref{worst_error_est}) suggests that
\begin{align*}
&\|\bp^{\ell+1} \tth\|^2_{L^2(\bbS^1)} + \rho^2 \rho^\pprime \|\bp^\ell \widetilde{\tth}\|^2_{L^2(\bbS^1)} + \int_0^t \rho^3 \|\bp^\ell \ttv\|^2_{L^2(B_1)} ds - \int_0^t \rho^2 \rho^{\pprime\prime} \|\bp^\ell \widetilde{\tth}\|^2_{L^2(\bbS^1)} ds \nonumber\\
&\qquad\quad \le \|\bp^\ell\tth \|^2_{L^2(\bbS^1)} + C_{\delta,\delta_1} \Nhh \\
&\qquad\qquad + C_\delta \int_0^t \rho^3 \|\ttv\|^2_{H^{\ell-1}(B_1)} ds + \delta \int_0^t \rho^3 \|\ttv\|^2_{H^\ell(B_1)} ds + \delta_1 \triplenorm{\tth}^2_T \nonumber
\end{align*}
for some constant $0<c<1$.

On the other hand, for $0 \le \ell \le \rK-1$, assumption (\ref{assumption1}) and condition (\ref{rho_assumption1}) or (\ref{rho_assumption2}) imply that
\begin{align*}
& \rho^2 \rho^\pprime \|\bp^\ell \tth\|^2_{L^2(\bbS^1)} \le 2 \rho^2 \rho^\pprime \Big[\|\bp^\ell \widetilde{\tth}\|^2_{L^2(\bbS^1)} + \frac{1}{4\rho^2} \|\bp^\ell (\tth^2)\|^2_{L^2(\bbS^1)} \Big] \\
&\qquad \le 2 \rho^2 \rho^\pprime \|\bp^\ell \widetilde{\tth}\|^2_{L^2(\bbS^1)} + C \rho^\pprime \|\tth\|^2_{H^1(\bbS^1)} \|\bp^\ell \tth\|^2_{L^2(\bbS^1)} \\
&\qquad \le 2 \rho^3 \rho^\pprime \|\bp^\ell \widetilde{\tth}\|^2_{L^2(\bbS^1)} + C \sigma \triplenorm{\tth}^2_T
\end{align*}
which, in turn,  implies that
\begin{align}
&\|\bp^{\ell+1} \tth\|^2_{L^2(\bbS^1)} + \frac{1}{2}\hspace{1pt} \rho^2 \rho^\pprime \|\bp^\ell \tth\|^2_{L^2(\bbS^1)} + \hspace{-2pt} \int_0^t \rho^3 \|\bp^\ell \ttv\|^2_{L^2(B_1)} ds - \hspace{-2pt}\int_0^t \rho^2 \rho^{\pprime\prime} \|\bp^\ell \widetilde{\tth}\|^2_{L^2(\bbS^1)} ds \nonumber\\
&\qquad\quad \le \|\bp^\ell\tth \|^2_{L^2(\bbS^1)} + C_{\delta,\delta_1} \Nhh \label{energy_est_temp1}\\
&\qquad\qquad + C_\delta \int_0^t \rho^3 \|\ttv\|^2_{H^{\ell-1}(B_1)} ds + \delta \int_0^t \rho^3 \|\ttv\|^2_{H^\ell(B_1)} ds + 2 \delta_1 \triplenorm{\tth}^2_T \nonumber
\end{align}
if $C \sigma$ is chosen to be smaller than $2 \delta_1$.

If $\ell = 1$, by (\ref{L2_energy_est}) we find that
\begin{align}
&\|\bp^2 \tth\|^2_{L^2(\bbS^1)} + \frac{1}{2}\hspace{1pt} \rho^2 \rho^\pprime \|\bp \tth\|^2_{L^2(\bbS^1)} + \int_0^t \rho^3 \|\bp \ttv\|^2_{L^2(B_1)} ds - \hspace{-2pt}\int_0^t \rho^2 \rho^{\pprime\prime} \|\bp \widetilde{\tth}\|^2_{L^2(\bbS^1)} ds \nonumber\\
&\qquad\quad \le C_{\delta_1} \Nhh + \delta \int_0^t \rho^3 \|\ttv\|^2_{H^\ell(B_1)} ds + 2\delta_1 \triplenorm{\tth}^2_T \nonumber
\end{align}
which, combined with Lemma \ref{lem:hodge} and (\ref{L2_energy_est}), implies that
\begin{align*}
&\|\tth\|^2_{H^2(\bbS^1)} + \frac{1}{2}\hspace{1pt} \rho^2 \rho^\pprime \|\tth\|^2_{H^1(\bbS^1)} + c \int_0^t \rho^3 \| \ttv\|^2_{H^1(B_1)} ds - \hspace{-2pt}\int_0^t \rho^2 \rho^{\pprime\prime} \|\widetilde{\tth}\|^2_{H^1(\bbS^1)} ds \nonumber\\
&\qquad\qquad\qquad \le C_{\delta_1} \Nhh + 2 \delta_1 \triplenorm{\tth}^2_T
\end{align*}
for some constant $0<c<1$. A similar argument, again relying on Lemma \ref{lem:hodge},  shows that
\begin{align*}
&\|\tth\|^2_{H^\rK(\bbS^1)} + \frac{1}{2}\hspace{1pt} \rho^2 \rho^\pprime \|\tth\|^2_{H^{\rK-1}(\bbS^1)} + c \int_0^t \rho^3 \|\ttv\|^2_{H^{\rK-1}(B_1)} ds - \hspace{-2pt} \int_0^t \rho^2 \rho^{\pprime\prime} \|\widetilde{\tth}\|^2_{H^{\rK-1}(\bbS^1)} ds \nonumber\\
&\qquad\qquad\qquad \le C_{\delta_1} \Nhh + C \delta_1 \triplenorm{\tth}^2_T.
\end{align*}

Finally, we look for an upper bound of $\displaystyle{}\int_0^t \rho^2 \rho^{\pprime\prime} \|\widetilde{\tth}\|^2_{H^{\rK-1}(\bbS^1)} ds$ to close the energy estimates. If (\ref{rho_assumption1}) is satisfied, by interpolation and the inequality $\|\widetilde{\tth}\|_{H^6(\bbS^1)} \le C \|\tth\|_{H^6(\bbS^1)}$,
\begin{align*}
& \Big|\int_0^t \rho(s)^2 \rho^\pprime(s) \|\widetilde{\tth}(s)\|^2_{H^5(\bbS^1)} ds\Big| \le C \int_0^\infty \frac{\rho(s)^3}{1+s} \|\widetilde{\tth}(s)\|^{4/7}_{H^{2.5}(\bbS^1)} \|\widetilde{\tth}(s)\|^{10/7}_{H^6(\bbS^1)} ds \\
&\quad \le C \Big[\hspace{-1pt}\int_0^\infty \hspace{-2pt}\frac{\rho(s)^{13/7}}{1+s} e^{-4\beta d(s)/7} ds \Big] \Nhh \hspace{-1pt}\le C \Nhh.
\end{align*}
Now suppose that (\ref{rho_assumption2}) is satisfied. If $\rho^{\pprime\prime} \le 0$, then
\begin{equation}\label{energy_estimate}
\begin{array}{l}
\displaystyle{}\|\tth(t)\|^2_{H^\rK(\bbS^1)} + \frac{1}{2}\hspace{1pt} \rho(t)^2 \rho^\pprime(t) \|\widetilde{\tth}(t)\|^2_{H^{\rK-1}(\bbS^1)} + \int_0^t \rho(s)^3 \|\ttv(s)\|^2_{H^{\rK-1}(B_1)} ds \vspace{.2cm}\\
\displaystyle{}\qquad\qquad\qquad \le C_{\delta_1} \Nhh + C \delta_1 \triplenorm{\tth}^2_T.
\end{array}
\end{equation}
If $\log\rho^\pprime$ has small total variation, then for some $\delta_2 \ll 1$,
\begin{align*}
\big\|(\log \rho^\pprime)^\pprime\big\|_{L^1(0,\infty)} = \int_0^\infty \Big|\frac{\rho^{\pprime\prime}(s)}{\rho^\pprime(s)}\Big| ds \le \delta_2.
\end{align*}
Since $\rho \sqrt{\rho^\pprime} \hspace{1pt}\|\widetilde{\tth}\|_{H^{\rK-1}(\bbS^1)} \le \triplenorm{\tth}_T$,
\begin{align*}
& \Big|\int_0^t \rho(s)^2 \rho^\pprime(s) \|\widetilde{\tth}(s)\|^2_{H^{\rK-1}(\bbS^1)} ds\Big| \le \Big[\int_0^\infty \frac{|\rho^{\pprime\prime}(s)|}{\rho^\pprime(s)} ds \Big] \triplenorm{\tth}^2_T \le \delta_2 \triplenorm{\tth}^2_T;
\end{align*}
thus we obtain (\ref{energy_estimate}) again with $\delta_1$ replaced by $\delta_1 + \delta_2$.

\begin{remark}
By the energy estimate {\rm(\ref{energy_estimate})}, we see that the quantity $\|\widetilde{\tth}\|_{H^{\rK-1}(\bbS^1)}$ decays at the rate $\displaystyle{}\frac{1}{\rho \sqrt{\rho^\pprime}}$. This decay rate is slower than the decay of $\|\tth\|_{H^{2.5}(\bbS^1)}$ if $\rho$ grows algebraically.  Thus, our approach of establishing
decay in the lower-order norm via
Duhamel's principle provides better decay than energy estimates alone.
\end{remark}

\section{The stability of the Hele-Shaw flow with injection}\label{sec:stability_result}
Combining (\ref{decay_estimate}), (\ref{elliptic_est}) and (\ref{energy_estimate}), by choosing $\delta_1 > 0$ small enough we conclude that
\begin{align*}
\triplenorm{h}^2_T \le C \Nhh
\end{align*}
or
\begin{align}
\triplenorm{\tth}_T \le C \Big[\|\tth_0\|_{H^\rK(\bbS^1)} + \triplenorm{\tth}^2_T + \triplenorm{\tth}^{2p}_T\Big] \label{energy_inequality}
\end{align}
for some $p\in \bbN$, provided that condition (\ref{rho_assumption1}) or (\ref{rho_assumption2}), as well as the assumption (\ref{assumption1}), are valid. 
The constant $ \sigma $ is chosen sufficiently small so as to absorb certain error terms on the right-hand side of our energy estimates by the left-hand side
energy terms.   In other words, as long as all of the constants we use in our elliptic estimates and the Sobolev embedding theorem  are fixed, the maximum of $\sigma$ is a fixed  computable number which can be chosen independent of $\|\tth_0\|_{H^6(\bbS^1)}$.  Inequality (\ref{energy_inequality}) then implies that there exists $\epsilon > 0$ small enough such that if $\|\tth_0\|_{H^\rK(\bbS^1)} \le \epsilon$, by the continuity (in time) of $\tth$,
\begin{align}
\triplenorm{\tth}_T \ll \sigma. \label{total_norm_bound}
\end{align}
This suggests that as long as the solution exists, $\tth$ has to satisfy the estimate (\ref{total_norm_bound}), and this establishes Theorem \ref{thm:main_thm1} and \ref{thm:main_thm2}.

\vspace{.1in}

\noindent {\bf Acknowledgments.} AC was supported by the National Science Council (Taiwan) under grant 100-2115-M-008-009-MY3.
DC was supported by the Centre for Analysis and Nonlinear PDEs funded by the UK EPSRC grant EP/E03635X and the Scottish Funding Council. 
SS was supported by the National Science Foundation under grant DMS-1001850.

\bibliography{Reference}
\bibliographystyle{plain}


\end{document}